\documentclass[a4,12pt]{article}
\usepackage{graphicx, bbding}
\usepackage{ mathrsfs, mathtools }  
\usepackage{psfrag}
\usepackage{subfigure}
\usepackage{amsmath}
\usepackage{amsfonts}
\RequirePackage[OT1]{fontenc}
\RequirePackage{amsthm,amsmath,amssymb,amstext,latexsym,epsfig}
\RequirePackage{txfonts}
\def\D{\mathcal{D}}
\def\pa{\mathrm{pa}}
\def\fa{\mathrm{fa}}
\def\pr{\mathrm{pr}}
\def\G{\mathcal{G}}

\theoremstyle{plain}
\newtheorem{theorem}{Theorem}[section]
\newtheorem{corollary}{Corollary}[section]

\newtheorem{lemma}{Lemma}[section] 
\newtheorem{proposition}{Proposition}[section]
\theoremstyle{definition}
\newtheorem{definition}{Definition}[section]
\theoremstyle{remark}

\numberwithin{equation}{section}
\newtheorem{Ex}{Example}[section]
\newtheorem{Rem}{Remark}[section]

\newcommand{\R}{\mathbb{R}}

\usepackage[tmargin=1.1in,bmargin=1.1in,rmargin=1.25in,lmargin=1.25in]{geometry}

\title{Positive definite completion problems for directed acyclic graphs}
\author{Emanuel Ben-David\\ {\small Stanford University} \and Bala Rajaratnam\\ {\small Stanford University}}

\date{}

\begin{document}

\maketitle

\begin{abstract}
A positive definite completion problem pertains to determining whether the unspecified positions of a partial (or incomplete) matrix can be completed in a desired subclass of positive definite matrices. In this paper we study an important and new class of positive definite completion problems where the desired subclasses are the spaces of covariance and inverse-covariance matrices of probabilistic models corresponding to directed acyclic graph models (also known as Bayesian networks). We provide fast procedures that determine whether a partial matrix can be completed in either of these spaces and thereafter proceed to construct the completed matrices. We prove an analog of the positive definite completion result for undirected graphs in the context of directed acyclic graphs, and thus proceed to characterize the class of DAGs which can always be completed. We also proceed to give closed form expressions for the inverse and the determinant of a completed matrix as a function of only the elements of the corresponding partial matrix.    
\end{abstract}


\vspace{0.2cm}


\section{Introduction}
\label{sec:intro}
A pattern of a given $n\times n$  matrix is defined to be a subset of $\left\{\{ i,j\}: \: 1\leq i,j\leq n\right\}$, i.e., a set of positions in the matrix in which the entries are present. A (symmetric) partial matrix specified by a pattern is an $n\times n$ symmetric matrix in which the entries corresponding to the positions listed in the pattern are specified, but the rest of the entries are unspecified and thus free to be chosen. For example
\[
\begin{pmatrix}
1& ?& ?& -1\\
?&-2&2&?\\
?&2&?&?\\
-1&?&?&1/2
\end{pmatrix}
\] 
is a $4\times 4$ partial matrix specified by the pattern 
\[
\left\{ \{ 1,1\},\{2 ,2\} , \{ 4,4\}, \{ 1,4\}, \{ 2,3\}  \right\}.
\]

  A matrix completion problem asks whether for a given pattern the unspecified entries of each partial matrix can be chosen in such a way that the resulting conventional matrix  is of a desired type. Recent literature in general, and in particular matrix theory, computer science, statistics and signal processing have studied a variety of matrix completion problems, such as positive definite completion \cite{Grone1984}, low-rank completion \cite{Candes2009} and singular value completion \cite{Cai2010}. \\
  
 The positive definite completion problem, one of the most well studied matrix completion problems, asks which partial matrices have positive definite completions, with or without additional features. It is clear that if a partial matrix has a positive definite completion, then it must be partial positive definite, i.e., each fully specified principal submatrix is positive definite. \\
 
 The work of Grone et al. \cite{Grone1984} is one of the important contributions in the area of positive definite completion. The authors prove that every partial positive matrix corresponding to a given pattern has a positive definite completion if and only if the specified pattern, considered as a set of edges, forms a chordal (or equivalently
decomposable) graph. A chordal graph is an undirected graph that has no induced cycle of length greater than or equal to 4. Although a positive definite completion is not necessarily unique, they furthermore prove that a positive definite completion  matrix $\Sigma$  is unique if one requires that $\Sigma_{ij}^{-1}=0$  for each unspecified position $\{i,j\}$. Interestingly, such a positive definite completion for $\Sigma$ arose in an earlier paper \cite{Dempster1972} by Dempster within the context of maximum likelihood estimation for Gaussian graphical models. In light of advances in the area of graphical models in recent years, the connection between positive definite completion problems and probabilistic models corresponding to undirected graphs has been thoroughly exploited (see for example in \cite{Dawid1993} \cite{Roverato2000}, \cite{Letac2007}, \cite{Rajarat2008}, \cite{Khare2011}). \\
 
In this paper we study a new class of positive definite completion problems that corresponds to probabilistic models over directed acyclic graphs, abbreviated DAG models henceforth. DAG models, better known as Bayesian networks, are arguably one of the most widely used classes of graphical models. The need for studying this new class of problems naturally arises when studying spaces of covariance and inverse covariance matrices corresponding to DAG models \cite{Bendavid2011}. These spaces are essential features of a DAG model. In the DAG setting, we consider specific positive definite completions of partial matrices that are specified by a pattern determined by the edges of a directed acyclic graph $\D$. Here the partial matrices are desired to be completed in the space of covariance or inverse-covariance matrices corresponding to DAG models. \\
 
 A great advantage of positive definite completion problems for DAGs, as we shall determine in this paper, is that we are able to present fast decision procedures to not only determine whether a given partial matrix can be completed in either of the aforementioned spaces above, but we are also able to fully construct the completed matrices. This highlights the tractable nature of completion problems for DAGs as compared with those of undirected graphs, where no such completion  procedures are known to exist, unless the graph is decomposable. Even in case of a decomposable undirected graph $\G$, we shall see that the same positive definite completion can be achieved under a directed version $\D$ of $\G$. In this sense our result aims to generalize the existing results for completing partial positive definite matrices corresponding to undirected graphs.\\
    
 The organization of the paper is as follows. In \S \ref{sec:pre} we briefly review the basics of  graphical models and in particular Gaussian graphical models for both undirected and directed graphs. In \S \ref{sec:pdcp} first we formally define two types of positive definite completion problems  for DAGs and then, among some other results, we present two fast procedures of polynomial complexity that determine whether a partial matrix can be completed in the desired space, and specify a way to uniquely construct the completed matrix. The uniqueness of this completion has tremendous benefits for Bayesian analysis of DAG models (see \cite{Bendavid2011} for more details). In \S \ref{sec:gen} we prove an analog of Grone et al.'s \cite{Grone1984} theorem  but in the context of DAGs, and also demonstrate subtle differences between the completion problems for DAGs vs. undirected graphs. In \S \ref{sec:det_inv} we provide expressions for directly computing the inverse and the determinant of the positive definite completion of a partial matrix, without actually carrying out the completion. 

 \section{Preliminaries}
 \label{sec:pre}
\subsection{Graph theoretic notation and terminology}\label{subsec:graph} A graph $\G$ is a pair of objects $(V, \mathscr{V})$, where $V$ and $\mathscr{V}$ are two disjoint finite sets representing,
respectively, the vertices and the edges of $\G$. Each edge $e\in \mathscr{V}$ is either an ordered pair $(\nu,\nu')$ or an unordered pair $\{ \nu,\nu'\}$, for some $\nu,\nu'\in V$. An edge $(\nu,\nu')\in \mathscr{V}$ is called directed where $\nu$ is said to be a parent of $\nu'$, and $\nu'$ is said to be a child of $\nu$. We write this as $\nu\rightarrow \nu'$. The set of parents of $\nu$ is denoted by $\mathrm{pa}(\nu)$, and the set of children of $\nu$ is denoted by $\mathrm{ch}(\nu)$. The family of $\nu$ is $\mathrm{fa}(\nu)=\mathrm{pa}(\nu)\cup \{\nu\}$. An edge $\{\nu,\nu'\}\in \mathscr{V}$ is called undirected where $\nu$ is said to be a neighbor of $\nu'$, or $\nu'$ a neighbor of $\nu$. We write this $\nu\sim_{\G}\nu'$. The set of all neighbors of $\nu$ is denoted by $\mathrm{ne}(\nu)$. We say $\nu$ and $\nu'$ are adjacent if there exists  either a directed or an undirected edge between them. The boundary of $\nu$, denoted by $\mathrm{bd}(\nu)$, is the union of parents and neighbors of $\nu$. A loop in $\G$ is an ordered pair $(\nu,\nu)$, or an unordered pair $\{\nu,\nu\}$ in $\mathscr{V}$. For ease of notation, in this paper we always shall assume that the edge set $\mathscr{V}$ contains all the loops, although we shall draw the respective graphs without the loops. \\

\indent An undirected graph is a graph with all of its edges undirected, whereas a directed graph, ``digraph", is a graph with all of its edges directed. In this section, we shall use the symbol $\G$ to denote a general graph, and make clear within the context in which it is used, whether $\G$ is  directed or undirected.\\

\indent We say that the graph $\mathcal{G}'=(V', \mathscr{V}')$ is a subgraph of  $\mathcal{G}=(V, \mathscr{V})$,  denoted by  $\mathcal{G}'\subseteq \mathcal{G}$,  if  $V'\subseteq V$  and  $ \mathscr{V}'\subseteq  \mathscr{V}$. In addition, if  $\mathcal{G}'\subseteq \mathcal{G}$  and  $ \mathscr{V}'=V'\times V'\cap  \mathscr{V}$, we say that $\mathcal{G}'$ is an {induced}  subgraph of  $\mathcal{G}$. We shall consider only induced subgraphs in what follows. For a subset  $A\subseteq V$, the induced subgraph  $\mathcal{G}_A=(A, A\times A\cap  \mathscr{V})$ is said to be the graph {induced} by $A$. A graph  $\mathcal{G}$  is called {complete} if every pair of vertices are adjacent. A  clique  of  $\mathcal{G}$  is an induced complete subgraph  of  $\mathcal{G}$ that is not a subset of any other induced complete subgraphs  of  $\mathcal{G}$. More simply,  a subset  $A\subseteq V$  is  called a clique if the induced subgraph  $\mathcal{G}_A$  is a clique of $\mathcal{G}$. The set of the cliques of $\G$ is denoted by $\mathscr{C}_{\G}$.\\

A path in $\G$ of length $n\geq 1$ from a vertex $\nu$ to a vertex $\nu'$ is a finite sequence of distinct vertices $\nu_{0}=\nu,\ldots, \nu_{n}=\nu'$ in $V$ such that $(v_{k-1}, v_k)$  or $\{ v_{k-1}, v_k\}$ are in $\mathscr{V}$ for each $k=1,\ldots, n$. We say that the path is {directed} if at least one of the edges is directed. We say  $v$  {leads} to $v'$, denoted by  $v\longmapsto v'$, if there is a directed path from  $v$  to  $v'$. A graph $\mathcal{G}=(V, \mathscr{V})$ is called connected if for any pair of distinct vertices $v, v'\in V$ there exists a path between them. An  $n$-{cycle} in  $\mathcal{G}$  is a path of length  $n$  with the additional requirement that the end points are identical. A {directed}  $n$-cycle is defined accordingly. A graph is {acyclic} if it does not have any cycles. An {acyclic directed  graph}, denoted by  DAG, is a directed graph with no cycles of length greater than 1.\\

\indent The undirected version of a graph  $\mathcal{G}=(V, \mathscr{V})$, denoted by $\mathcal{G}^{\mathrm{u}}$,  is the undirected graph obtained by replacing all the directed edges of  $\mathcal{G}$  by undirected ones. An immorality in a directed  graph  $\mathcal{G}$  is  an induced subgraph of  the form  $i\longrightarrow k \longleftarrow j$.  Moralizing an immorality entails adding an undirected edge between the pair of parents that have the same children. Then the moral graph of  $\mathcal{G}$, denoted by  $\mathcal{G}^{\mathrm{m}}$, is the undirected graph obtained by first  {moralizing} each  immorality of  $\mathcal{G}$  and then making the undirected version of the resulting graph. Naturally there are DAGs which have no immoralities and this leads to the following definition.

\begin{definition}
A DAG   $\mathcal{G}$ is said to be ``perfect" if it has no immoralities; i.e., the parents of all vertices are adjacent, or equivalently if the set of parents of each vertex induces a complete subgraph of  $\mathcal{G}$ .
\end{definition}

\indent Given a directed acyclic graph (DAG), the set of {ancestors} of a vertex $v$, denoted by $\mathrm{an}(v)$, is the set of those vertices $v''$  such that $v''\longmapsto v$. Similarly, the set of  {descendants}  of  a vertex $v$, denoted by  $\mathrm{de}(v)$, is the set of  those  vertices  $v'$  such that  $v\longmapsto v'$.  The   set of {non-descendants } of  $v$  is  $\mathrm{nd}(v)=V\setminus\left(\mathrm{de}(v)\cup \{v\}\right)$.

An undirected graph $\mathcal{G}$ is said to be decomposable if no induced subgraph contains  a cycle of length greater than or equal to four. A constructive definition in terms of the cliques and the separators of the graph $\G$ can also be specified. The reader is referred to Lauritzen \cite{Lauritzen1996}  for all the common notions of decomposable graphs that we will use here. 
\noindent
Decomposable (undirected) graphs and (directed) perfect graphs have a deep connection. In particular, it can be shown \cite{Golumbic2004, Lauritzen1996} that if $\G$ is decomposable, then there exists a directed version of $\G$, i.e., a digraph $\D$ such that $\D^{\mathrm{u}}=\G$, where $\D$ is a perfect DAG.

\subsection{Graphical Gaussian models} 
A graphical model over a graph $\G$ is a family of probability distributions on a common probability space such that each distribution satisfies the set of conditional independences described by $\G$. Two important classes of graphical models are Markov random fields (or undirected graphical models)  and  Bayesian networks (or directed graphical models). Henceforth in this paper, we shall assume that $\G=(V, \mathscr{V})$ is an undirected graph and $\D=(V, \mathscr{E})$ is a directed acyclic graph (DAG), both with the same vertex set  $V=\{1,\ldots, p\}$. A random vector
$\mathbf{X}=(X_{1},\ldots, X_{p})\in \R^{p}$ belongs to a Markov random field over $\G$ if it satisfies the pairwise Markov property\footnote{More precisely, a  Markov random field is a probability distribution that satisfies the local Markov property, which is in general a stronger property than the pairwise Markov property. However, when $P$ has a positive density w.r.t. a Borel measure these two properties are equivalent.} w.r.t.  $\G$, i.e.,
\begin{equation}\label{eq:GM}
(i,j)\notin\mathscr{V}\Longrightarrow X_{i}\Perp X_{j}| \left(X_{k}: k\in V\setminus\{ i,j\}\right).
 \end{equation}
For a set  $A\subseteq V$ let  $\mathbf{X}_{A}=(X_{i}: \; i\in A)$.  The random vector  $\mathbf{X}$  belongs to  a Bayesian network over  a DAG   $\D$  if it satisfies the directed local  Markov property w.r.t.  $\D$, i.e.,
\begin{equation}
X_{i}\Perp \mathbf{X}_{\mathrm{nd}(i)}| \mathbf{X}_{\mathrm{pa}(i)}\quad\forall i\in V.
\end{equation}
\begin{figure}[ht]
\centering
\subfigure[]{
\includegraphics[width=2.4cm]{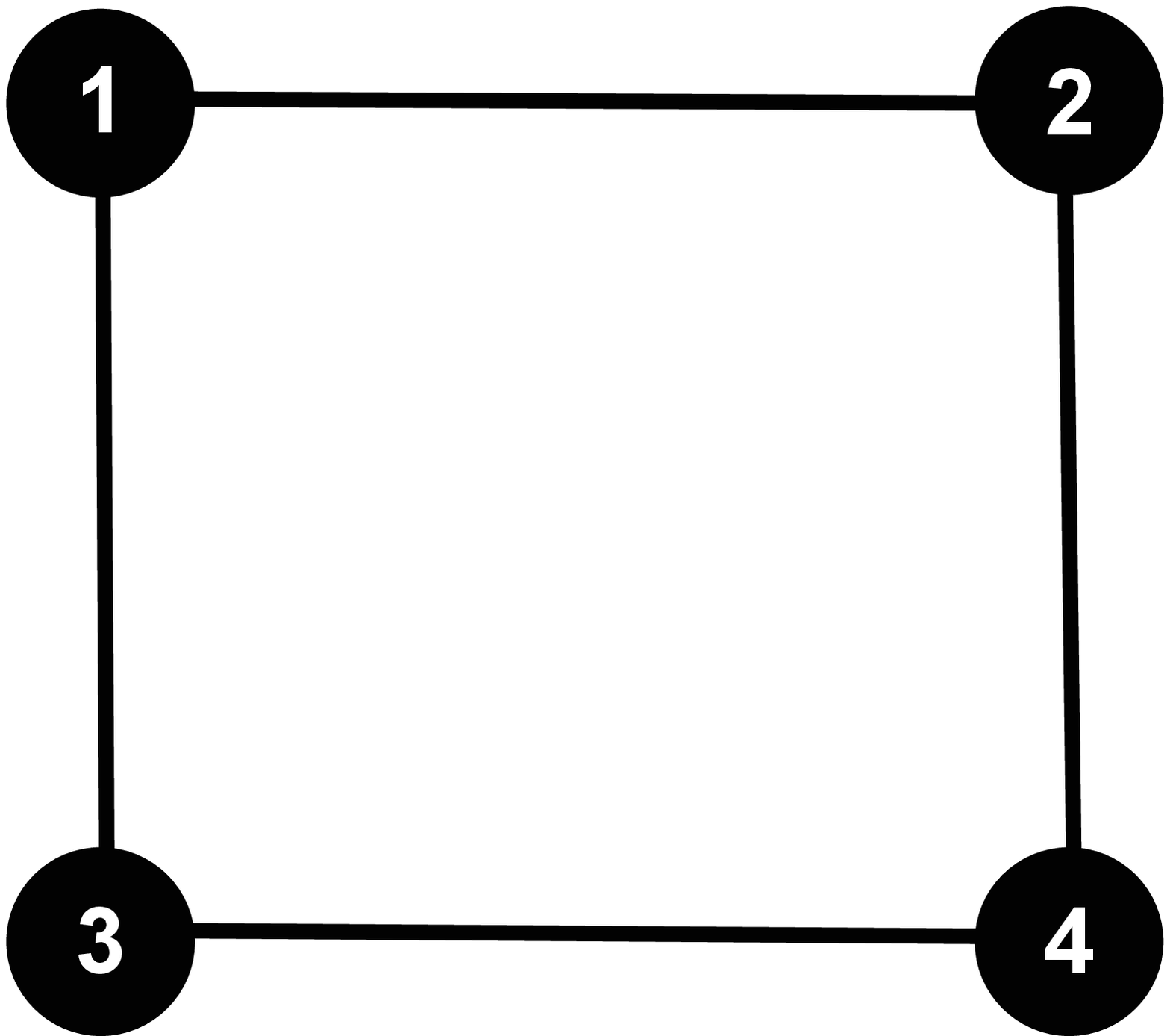}
\label{fig:C4}
}
\hspace{3cm}
\subfigure[]{
\includegraphics[width=2.4cm]{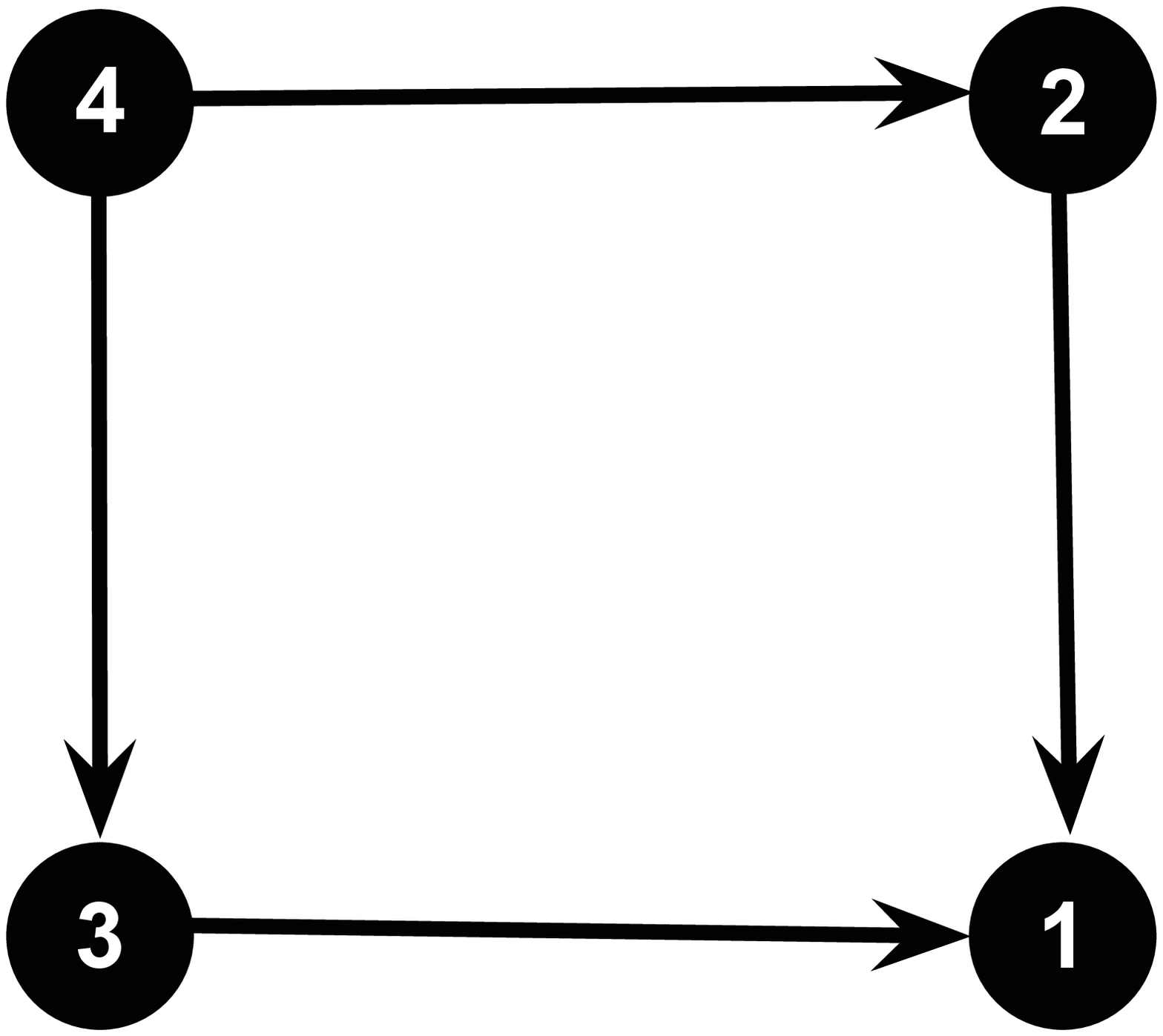}
\label{fig:ddC4}
}
\caption{ An undirected four cycle and a directed version. }
\end{figure}
We now give examples of such models.

\begin{Ex}\
\begin{itemize}
\renewcommand{\labelitemi}{$(a)$}
\item Let us consider the four cycle $C_{4}$ in Figure \ref{fig:C4}. Then a random vector $(X_{1},\ldots, X_{4})\in \R^{4}$ satisfies the global Markov property w.r.t. $C_{4}$ implies
\[
X_{1}\Perp X_{4}|( X_{2}, X_{3})\quad\text{and}\quad X_{2}\Perp X_{3}| (X_{1}, X_{4}).
\]
\renewcommand{\labelitemi}{$(b)$}
\item Let us consider the DAG  $\D$ given in Figure \ref{fig:ddC4}. Note that $\D$ is a directed version of $C_{4}$. Now a random vector  $(X_{1},\ldots, X_{4})$ satisfies the directed Markov property w.r.t. $\D$ implies
\[
\quad X_{1}\Perp X_{4}|(X_{2}, X_{3}) \quad\text{and} \quad X_{2}\Perp X_{3}| X_{4}.
\]
\end{itemize}
\end{Ex}

Important subclasses of the models defined above arise when $\mathbf{X}$ is multivariate Gaussian; namely the Gaussian Markov random field over $\G$ and the directed Gaussian random field over $\D$, denoted by $\mathscr{N}(\G)$  and   $\mathscr{N}(\D)$, respectively. More precisely, $\mathscr{N}(\G)$  denotes the family of multivariate normal distributions $\mathrm{N}_{p}(\mu,\Sigma)$,  $\mu\in \R^{p}$,  $\Sigma\in\mathrm{PD}_{p}(\R)$ (abbreviated  $\Sigma\succ 0$ henceforth), that obey the local Markov property w.r.t. $\G$. The family of distributions $\mathscr{N}(\G)$ is often referred to as the Gaussian graphical or Gaussian undirected model over $\G$. Likewise, $\mathscr{N}(\D)$ denotes the family of  multivariate normal distributions $\mathrm{N}_{p}(\mu,\Sigma)$ that obey the directed local Markov property w.r.t. $\D$. We shall refer to $\mathscr{N}(\D)$ as the Gaussian directed acyclic graph or DAG model over $\D$. It turns out that in both of these Gaussian graphical models, the required Markov property is reflected in $\Sigma$ in terms of  certain algebraic equations of the entries of $\Sigma$ depending on the structure of the underlying graph. The description of the equations  in  $\mathscr{N}(\G)$  is quite simple. We have  $\mathrm{N}_{p}(\mu, \Sigma)\in \mathscr{N}(\G)$ if and only if
\begin{equation}\label{eq:graphical_model_prop}
\{i,j\}\notin \mathscr{V}\Longrightarrow \Sigma_{ij}^{-1}=0.
\end{equation}
In order to express the equations satisfied by  $\Sigma$  when  $\mathrm{N}_{p}(\mu, \Sigma)\in \mathscr{N}(\D)$,  first we establish a few (and to some degree standard) notations for DAG models \cite{Andersson1998, Bendavid2011}. \\

\noindent
First note that the relation $\nu\rightarrow \nu'$ defines a partial order on the  vertex set  of  $\D$. Since any partial order can be extended to a total order, we can therefore assume without loss of generality that the vertices are numbered in such a way that $i\rightarrow j$ implies that $i>j$, $\forall i,j \in V$. By this convention a random vector  $\mathbf{X}\in \R^{p}$ obeys the directed local Markov property w.r.t. $\D$ (or more precisely an equivalent version of it called the ordered Markov property - see \cite{Cowell99} for details) if  and only if   
\begin{equation}\label{eq:dlmp}
X_{j}\Perp \mathbf{X}_{\pr(j)}|\mathbf{X}_{\pa(j)} \:\: \forall j\in V,
\end{equation}
where  $\pr(j) =\left\{ i : i>j  \: ,\:  i\notin \pa(j)\right\}$ is called the set of predecessors of $j$. \\

\noindent
 Now for $A, B \subseteq  V$ let  $\Sigma_{A, B}$ denote the $|A|\times |B|$ submatrix of $\Sigma$ with rows indexed by $A$ and columns indexed by $B$. We often write $\Sigma_{A}$ for the principal submatrix $\Sigma_{A, A}$,  $\Sigma_{A, j}$ for $\Sigma_{A, \{ j\}}$ and similarly $\Sigma_{j, B}$ for $\Sigma_{\{ j\}, B}$.  By a result in \cite{Andersson1998} $\mathrm{N}_{p}(\mu,\Sigma)\in \mathscr{N}(\D)$ if and only if  $\Sigma\succ 0$\;  and
 \begin{equation}\label{eq:DL_normal}
 \Sigma_{\pr(j), j}=\Sigma_{\pr(j),\pa(j)}(\Sigma_{\pa(j)})^{-1}\Sigma_{\pa( j), j}\quad\forall j\in V,
 \end{equation}
Note that if $\mathrm{pa}(j)=\emptyset$, then Equation \eqref{eq:dlmp} implies that  $X_{j}\Perp \mathbf{X}_{\pr( j)}$. In particular,  this implies that  $\Sigma_{\pr(j), j}=0$. To include this situation in Equation  \eqref{eq:DL_normal} we shall use the convention $\Sigma_{\pa(j)}=1$ and $\Sigma_{\pa(j), j}=0$. Equation \eqref{eq:DL_normal} also illustrates how the directed local Markov property is reflected in the entries of $\Sigma$ in terms of algebraic equations. We now proceed to define the Schur complement of a symmetric positive definite matrix. Consider a $p \times p$ symmetric matrix $M$ partitioned as follows:
\[
M = \left(\begin{array}{cc}M_I & M_{IJ} \\ M_{JI} & M_J\end{array}\right)
\]
where $\{I, J\}$ is a partition of $\{1,2,\ldots,p\}$. Note that the matrix $M$ is positive definite if and only if $M_I$ is positive definite and $M_{J|I} = M_J - M_{JI}M_{I}^{-1}M_{IJ}$ is positive definite. The matrix $M_{J|I}$ is called the \emph{Schur complement} of $M_I$ in $M$. 

\section{The positive definite completion problem for DAGs}\label{sec:pdcp}
 In this section we propose two polynomial time procedures involving rational functions for completing partial positive definite matrices  to positive definite matrices that correspond to Gaussian Bayesian networks  $\mathscr{N}(\D)$. In addition, these completion problems, as we will explain later, are also generalizations of the classical positive definite problem in \cite{Grone1984}. To formalize the completion problem that we discuss in this paper, we introduce some definitions and notation.

\subsection{Preliminaries}
 Let $\D=(V, \mathscr{E})$ be a DAG. A $\D$-partial matrix  is a symmetric function
 \[
 \Gamma:\left\{(i,j): \{i,j\}\in \mathscr{V}\right\}\rightarrow \R\: \: \text{such that}\:\:  \Gamma_{ij}=\Gamma((i,j))=\Gamma_{ji}\quad\forall \{i,j\}\in\mathscr{V},
 \]
 where $\mathscr{V}$ is the edge set of the undirected version of the DAG $\mathcal{D}$.
 The set of all $\D$-partial matrices, denoted by $\mathrm{I}_{\D}$, is a real linear space of dimension $|\mathscr{E}|$. Recall that $\mathscr{C}_{\D}$ denotes the set of cliques of ${\D}$. Now for each clique $C\in \mathscr{C}_{\D}$, the restriction of  $\Gamma$ to  $C$, denoted by  $\Gamma_{C}$, is a $|C|\times |C|$ matrix  $(\Gamma_{ij})_{i,j\in C}$.  A partial positive definite matrix over  $\D$ is a $\D$-partial matrix  $\Gamma$  such that  $\Gamma_{C}\succ 0$  for each  $C\in\mathscr{C}_{\D}$. The set of all partial positive definite matrices over  $\D$  is denoted by  $\mathrm{Q}_{\D}$. A partial matrix over an undirected graph $\G$, or a partial positive definite matrix over an undirected graph $\G$, can be similarly defined. Next we define two sets; the set of  covariance matrices and the set of inverse-covariance matrices corresponding to a Gaussian Bayesian network  $\mathscr{N}(\D)$. More precisely, these spaces are, respectively: 
  \[
  \mathrm{PD}_{\D}=\left\{\Sigma:  \mathrm{N}_{p}(0,\Sigma)\in \mathscr{N}(\D)\right\}\quad\text{and}\quad\mathrm{P}_{\D}=\left\{\Omega :  \Omega^{-1}\in \mathrm{PD}_{\D}\right\}.
  \]
 Similarly, for an undirected graph $\G=(V, \mathscr{V})$ we define 
 \[
 \mathrm{PD}_{\G}=\left\{\Sigma:  \mathrm{N}_{p}(0,\Sigma)\in \mathscr{N}(\G)\right\}\quad\text{and}\quad\mathrm{P}_{\G}=\left\{\Omega :  \Omega^{-1}\in \mathrm{PD}_{\G}\right\}.
 \]
 \begin{Rem}
 Note that by Equation \eqref{eq:graphical_model_prop}  $\Sigma\in \mathrm{PD}_{\G}$ if and only if $\Sigma_{ij}^{-1}=0$ whenever $\{i,j\}\notin \mathscr{V}$.
 \end{Rem}
 A characterizing feature of a Gaussian Bayesian network  $\mathscr{N}(\D)$ is that the  structure of the underlying DAG (i.e., the graph itself), in terms of the missing arrows, can be fully recovered from the lower triangular matrix in the Cholesky decomposition of  $\Omega = \Sigma ^{-1}\in \mathrm{P}_{\D}$. The following remark formalizes this fact.
   
\begin{Rem}\label{rem:chol_decom_PD}
    Let $\mathcal{L}_{\D}$  denote the linear space  of all lower triangular matrices with unit diagonal entries such that
   \[
   L\in \mathcal{L}_{\D} \Longrightarrow L_{ij}=0 \:\:\text{for each}\:\; (i,j)\notin\mathscr{E}.
   \]
    Then  $\Omega\in\mathrm{P}_{\D}$ if and only if  there exists a lower triangular matrix  $L\in\mathcal{L}_{\D}$ and a diagonal matrix  $\Lambda$, with strictly positive diagonal entries, such that in the modified Cholesky decomposition  $\Omega=L\Lambda L'$ \cite{Wermuth1980, Andersson1998, Bendavid2011}. In addition, the modified Cholesky decomposition is unique.
    \end{Rem} 
We now discuss the relationships between the spaces $\mathrm{PD}_{\D}$ (or $\mathrm{P}_{\D}$) and its undirected counterpart $\mathrm{PD}_{\D^{\mathrm{u}}}$ (or $\mathrm{P}_{\D^{\mathrm{u}}}$). In particular, the modified Cholesky decomposition property of $\Omega\in \mathrm{P}_{\D}$ in Remark \ref{rem:chol_decom_PD} above implies the following.
\begin{lemma}[Wermuth \cite{Wermuth1980}]\label{lem:wermuth}
Suppose  $\D$  is an arbitrary DAG.  Then $\mathrm{PD}_{\D}\subseteq \mathrm{PD}_{\D^{\mathrm{m}}}$, where the undirected graph $\D^{\mathrm{m}}$ is the moral graph of $\D$.
\end{lemma}
\begin{proof}
Suppose $\Sigma\in \mathrm{PD}_{\D}$. Let $\Omega=\Sigma^{-1}=L\Lambda L'$ be the modified Cholesky decomposition of $\Omega\in \mathrm{P}_{\D}$. It is required to prove that $\Omega$ is also in $\mathrm{P}_{\D^{\mathrm{m}}}$. In light of Equation \ref{eq:graphical_model_prop} it suffices to show that if $i$ and $j$  are non-adjacent in $\D^{\mathrm{m}}$, then $\Omega_{ij}= 0$. On the contrary, suppose that $\Omega_{ij}\neq 0$. Therefore
\[
\Omega_{ij}=\sum_{k=1}^{p}\Lambda_{kk}L_{ik}L_{jk}\neq 0. 
\]
since $\Lambda_{kk}>0$, this implies that there exists $k$  such that $L_{ik}\neq 0$ and $L_{jk}\neq 0$. Consequently, $i\rightarrow k$ and $j\rightarrow k$. Hence $i$ and $j$ are parents of $k$ which in turn implies that $i$  and $j$  are adjacent in $\D^{\mathrm{m}}$, yielding a contradiction. 
\end{proof}

\begin{proposition}\label{prop:equiv}
If $\D$  is a perfect DAG, and $\G = \D^{\mathrm{u}}$ is the undirected version of $\D$, then $\mathrm{PD}_{\D}=\mathrm{PD}_{\G}$. 
\end{proposition}
\begin{proof}
Suppose $\D$ is a perfect DAG. Thus $\D^{\mathrm{m}}=\G$  and by Lemma \ref{lem:wermuth} $\mathrm{PD}_{\D}\subseteq \mathrm{PD}_{\G}$. Now to establish the other inclusion, assume that $\Sigma\in \mathrm{PD}_{\G}$.  Let $\Sigma^{-1}=L\Lambda L'$ be the modified Cholesky decomposition of $\Sigma^{-1}$. By Remark \ref{rem:chol_decom_PD} it suffices to show that if $(i,j)\notin \mathscr{E}$  and $i>j$, then $L_{ij}=0$. Note however that $(i,j)\notin \mathscr{E}$ implies that $\{i, j\}\notin \mathscr{V}$. Therefore
\begin{equation*}\label{eq:l_ij}
0=\Sigma_{ij}^{-1}=\Lambda_{jj}L_{ij}+\sum_{k\neq j}\Lambda_{kk}L_{ik}L_{jk}.
\end{equation*}
Assume to the contrary that $L_{ij}\neq 0$, then there exists an index $k$  such that $L_{ik}\neq 0$  and $L_{jk}\neq 0$. This in turn implies that there exists an immorality $i\rightarrow k\leftarrow j$ since by assumption $(i,j)\notin \mathscr{E}$. We have thus reached a contradiction to the fact that $\D$ is perfect. Therefore $L\in \mathcal{L}_{\D}$ and consequently $\Sigma^{-1}\in \mathrm{P}_{\D}$ or $\Sigma\in \mathrm{PD}_{\D}$.
\end{proof}
\begin{Rem}
Note that the statement of Proposition \ref{prop:equiv} can be rephrased as follows: if $\D$ is a perfect DAG, then a normal distribution obeys the directed local Markov property w.r.t. $\D$ if and only if it obeys the pairwise Markov property w.r.t. $\D^{\mathrm{u}}$, the undirected version of $\D$. It can be easily shown that if $\D$ is a perfect DAG, the above statement holds in more generality than just for normal distributions \cite{Lauritzen1996}.
\end{Rem}

\noindent {\sc Convention:} Hereafter in this paper, and unless otherwise stated, we assume that  $\G=(V,\mathscr{V})$  is the undirected version of  the DAG  $\D=(V,\mathscr{E})$, i.e., $\G=\D^{\mathrm{u}} $.

  \begin{definition}
  Let  $\mathcal{M}$  be a subset of the space of $p\times p$ symmetric matrices, denoted by $\mathrm{S}_{p}(\R)$. We say that a  $\D$-partial matrix  $\Gamma$  can be completed in $\mathcal{M}$ if there exists a matrix  $T\in \mathcal{M}$  such that   $T_{ij}=\Gamma_{ij}$  for each  $(i,j)\in \mathscr{E}$. We refer to $T$ as a completion of $\Gamma$ in $\mathcal{M}$, or simply a completion of $\Gamma$ if $\mathcal{M}$ is $\mathrm{S}_{p}(\R)$.
  \end{definition}
  
 Similar definitions can also given in the context of undirected graphs, i.e, the completion of $\G$-partial matrices.\\

  \begin{corollary}\label{cor:per_dec}
Let  $\D$ be a perfect DAG and let $\G = \D^{\mathrm{u}}$ denote the undirected version of $\D$. Then $\Gamma\in \mathrm{I}_{\D}$  can be completed in $\mathrm{PD}_{\D}$ if and only if  it can be completed in $\mathrm{PD}_{\G}$. 
\end{corollary}

\begin{proof}
The proof is immediate from Proposition \ref{prop:equiv}.
\end{proof}

   \subsection{Positive definite completion in $\mathrm{P}_{\D}$} \label{subsec:p_g}
An important question in the probabilistic analysis of directed Markov random fields/DAGs is whether a $\D$-partial matrix can be completed in $\mathrm{P}_{\D}$, i.e., whether a given $\D$-partial matrix corresponds to an inverse covariance matrix of a DAG model. Similar questions can be asked about completions in $\mathrm{PD}_{\D}$. These are inherently algebraic questions. Note that Remark \ref{rem:chol_decom_PD} implies that one can potentially recover the full matrix $\Omega$ merely from the entries that correspond to the edge set of  $\D$. We formalize this statement in the proposition below.
   
      \begin{proposition}\label{prop:completion_in_PD}
Let $\Gamma$ be a $\D$-partial matrix in $\mathrm{I}_{\D}$. If \ $\Gamma_{11}\neq 0$, then
\begin{itemize}
\renewcommand{\labelitemi}{$(a)$}
\item Almost everywhere (w.r.t. Lebesgue measure on $\mathrm{I}_{\D}$), there exist a unique lower triangular matrix $L\in \mathcal{L}_{\D}$  and a unique diagonal matrix $\Lambda \in \R^{p\times p}$ such that 
$\widehat{\Gamma}=L\Lambda L'$ is a completion of $\Gamma$.
\renewcommand{\labelitemi}{$(b)$}
\item The matrix $\widehat{\Gamma}$ is the unique positive definite completion of  $\Gamma$ in $\mathrm{P}_{\D}$ if and only if the diagonal entries of $\Lambda$ are all strictly positive.
\end{itemize}
\end{proposition}
\begin{proof}
~$(a)$~First we shall show that, almost everywhere w.r.t. Lebesgue measure on $\mathrm{I}_{\D}$,  $\Gamma$ can be uniquely completed to a matrix  $\widehat{\Gamma}$  in $\mathrm{S}_{p}(\R)$, not necessarily positive definite, such that   $\widehat{\Gamma}=L\Lambda L'$, for some $L\in \mathcal{L}_{\D}$  and  a diagonal matrix  $\Lambda \in R^{p\times p}$. We shall use the Cholesky factorization algorithm  \cite{Watkins1991} to construct  $\Lambda$ and $L$, column by column, in the following steps.
\begin{itemize}
\renewcommand{\labelitemi}{$ i)$}
\item Set $L_{ij}=0$ for each $(i,j)\notin \mathscr{E}$.
\renewcommand{\labelitemi}{$ ii)$}
\item  Set $\Lambda_{11}=\Gamma_{11},\; L_{i1}=\Lambda^{-1}_{11}\Gamma_{i1}$ for each $i\in \mathrm{pa}(1)$ and set $j=1$.
\renewcommand{\labelitemi}{$iii)$}
\item If $j<p$, then set $j=j+1$ and proceed to step  $iv)$, otherwise $L$ and $\Lambda$ are constructed such that they satisfy  the condition in part $(a).$
\renewcommand{\labelitemi}{$iv)$}
\item Set  $\Lambda_{jj}=\Gamma_{jj}-\displaystyle{\sum_{ k=1}^{ j-1}\Lambda_{kk}L^2_{jk}}$ and proceed to the next step.
\renewcommand{\labelitemi}{$ v)$}
\item  For each $i\in\mathrm{pa}(j)$ if  $\Lambda_{jj}\neq 0$, then set  $\displaystyle{L_{ij}=\Lambda_{jj}^{-1}( \Gamma_{ij}-\sum_{k=1}^{j-1}\Lambda_{kk}L_{ik}L_{jk} )}$, and return to step $ iii)$. If $\Lambda_{jj}=0$, then no completion of $\Gamma$  exists that satisfies the condition in part $(a)$. Consequently, $\Gamma$ cannot also be completed in $\mathrm{P}_{\D}$.
\end{itemize}
\noindent
 Note that for each $1\leq j \leq p$, the expression for $\Lambda_{jj}$ given by $\displaystyle{\Gamma_{jj}-\sum_{k=1}^{j-1}\Lambda_{kk}L^2_{jk}}$ in step $iv)$, considered as a function of $\Gamma_{ij}, \; (i,j)\in \mathscr{E}$, is a rational function, say  $p_{j}(\Gamma) $. In particular, almost everywhere w.r.t. Lebesgue measure on $\mathrm{I}_{\D}$, $\Lambda_{jj} = p_{j}(\Gamma)\neq 0$  for all $1\leq j\leq p$. Therefore, almost everywhere, the process above yields matrices $L$  and $\Lambda$ that satisfy the condition in part $(a)$.\\
\noindent
$(b)$ Part $(b)$ now follows from part $(a)$ and Remark \ref{rem:chol_decom_PD}.
\end{proof}

\begin{Rem}\label{rem:no_guarantee_PD}
It is clear from Proposition \ref{prop:completion_in_PD} above that a positive definite completion of  $\Gamma$ in $\mathrm{P}_{\D}$ is not always guaranteed. Moreover, the ability to complete in $\mathrm{P}_{\D}$ is not known beforehand, and is determined as a byproduct of having gone through the completion process itself. Having said this, if at any stage a $\Lambda_{jj}$ becomes negative or zero, it is evident from Proposition \ref{prop:completion_in_PD} that a completion in $\mathrm{P}_{\D}$ is no longer possible and the completion process can be terminated.   
\end{Rem}

\noindent We now demonstrate the completion outlined above on a partial matrix $\Gamma\in \mathrm{I}_{\D}$.

\begin{figure}[htbp]
\begin{center}
\includegraphics[width=3.0cm]{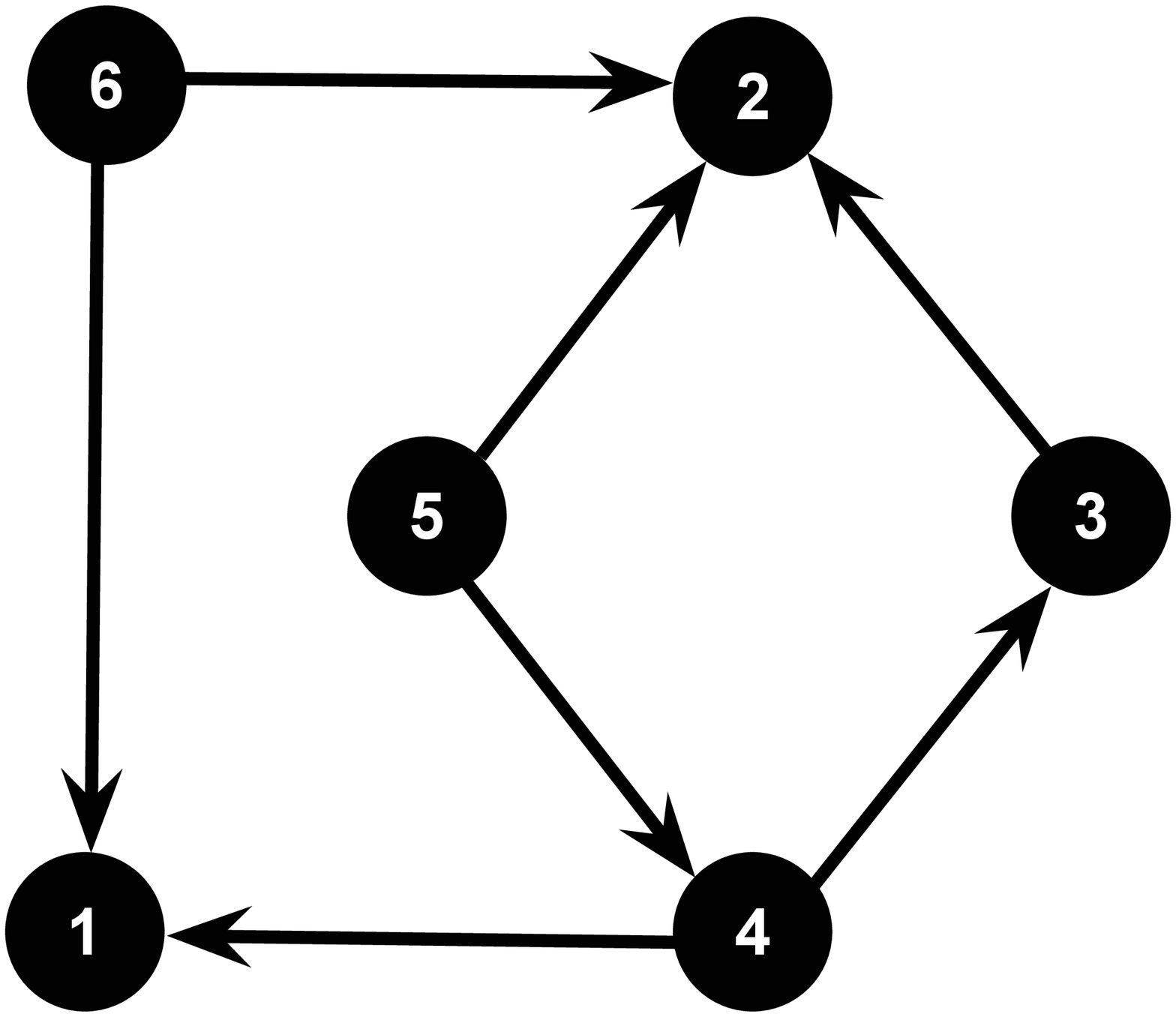}
\caption{Completion in  $\mathrm{P}_{\D}$.}
\label{fig:dag_pg}
\end{center}
\end{figure}

\begin{Ex}
Consider the DAG $\D$  given by Figure  \ref{fig:dag_pg}  and the $\D$-partial matrix $\Gamma$ given by
\[
\Gamma=
\left(
\begin{matrix}
 1  &  * &  * &  -3  &  *   & 4\\
  * &  -1 &  -2 &   * & -5  &  2\\
* & -2   &-2  &-10  &*  &  *\\
 -3  &  * & -10&   56    &3 & *\\
* &  -5  &*  &  3 & -30  & *\\
 4 &   2   & *  &*  &*  & 13
\end{matrix}
\right).
\]

Now by applying the completion process in Proposition \ref{prop:completion_in_PD} to $\Gamma$ we obtain
\[
\Lambda=
\left(
\begin{matrix}
 1    &0 &   0 &   0 &   0 &   0\\
  0   &-1 &   0 &   0 &   0 &   0\\
  0 &   0  &  2 &   0  &  0  &  0\\
  0    &0 &   0 &  -3 &   0 &   0\\
0   & 0 &   0 &   0 &  -2 &   0\\
 0    & 0 &   0  &  0  &  0 &   1
\end{matrix}
\right),\qquad 
L=
\left(
\begin{matrix}
1&    0 &   0 &   0 &   0 &   0\\
 0  &  1 &   0 &   0  &  0 &   0\\
 0  &  2 &   1 &   0 &   0 &   0\\
 -3 &   0 &  -5 &   1 &   0 &   0\\
  0  &  5 &   0 &  -1 &   1 &   0\\
4 &  -2 &   0 &   0  &  0  &  1
\end{matrix}
\right),
\]
which yields the completed matrix $\widehat{\Gamma}$ given as follows:
\[
\widehat{\Gamma}=
\left(
\begin{matrix}
 1  &  0  &  0 &  -3  &  0   & 4\\
  0 &  -1 &  -2 &   0  & -5  &  2\\
0  & -2   &-2  &-10  &-10  &  4\\
 -3  &  0 & -10&   56    &3 & -12\\
0 &  -5  &-10  &  3 & -30  & 10\\
 4 &   2   & 4  &-12   &10  & 13
\end{matrix}
\right).
\]
However, as the diagonal elements of $\Lambda$ are not strictly positive, $\Gamma$ cannot be completed in  $\mathrm{P}_{\D}$.
\end{Ex}

\subsection{Positive definite completion in $\mathrm{PD}_{\D}$}
\noindent An equally important question is whether a $\D$-partial matrix can be completed in $\mathrm{PD}_{\D}$, the space of covariance matrices corresponding to the DAG model $\mathscr{N}(\D)$. Recall that from Equation \eqref{eq:DL_normal} we have
\begin{equation}\label{eq:DL_normal_repeat}
\Sigma\in \mathrm{PD}_{\D}\Longleftrightarrow \: \Sigma\succ 0\: \text{and}\:\: \Sigma_{\pr(j),j}=\Sigma_{\pr(j), \pa(j)}(\Sigma_{\pa(j)})^{-1}\Sigma_{\pa(j),j} \:\; \forall j\in V.
\end{equation}
By recursively applying Equation \eqref{eq:DL_normal_repeat} we show below that it is possible to determine whether a $\D$-partial matrix can be completed in $\mathrm{PD}_{\D}$. The procedure is described in the following proposition .

\begin{proposition}\label{prop:completion_in_PDD}
Let $\Gamma\in\mathrm{Q}_{\D}$, then 
\begin{itemize}
\renewcommand{\labelitemi}{(a)}
\item There exists a completion process of polynomial complexity that can determine whether $\Gamma$ can be completed  in $\mathrm{PD}_{\D}$;  
\renewcommand{\labelitemi}{(b)}
\item If a completion exists, this completion is unique and can be determined constructively using the following process:\\
\end{itemize}
~$(1)$~ Set $\Sigma_{ij}=\Gamma_{ij}$ \; for each $\{i,j\}\in \mathscr{V}$  and set  $j=p$.\\
~$(2)$~ If  $j>1$, then set $j=j-1$ and proceed to the next step, otherwise $\Sigma$ is successfully completed.\\
~$(3)$~ If $\Sigma_{\fa(j)}\succ 0$, then proceed\footnote{Note that for each $j$, the submatrix $\Sigma_{\fa(j)}$ is fully determined by step {\it(2)}.} to the next step, otherwise the completion in $\mathrm{PD}_{\D}$  does not exist.\\
~$(4)$~ If $\pr(j)$ is empty, then return to step $(2)$, otherwise proceed to the next step.\\
~$(5)$~ If $\pa(j)$ is non-empty, then set $\Sigma_{\pr(j), j}=\Sigma_{\pr(j), \pa(j)}(\Sigma_{\pa(j)})^{-1}\Sigma_{\pa(j), j}$,  $\Sigma_{ j, \pr(j)}=\Sigma_{\pr(j), j }'$ and return to step $(2)$. If $\pa(j)$ is empty, then set $\Sigma_{\pr(j), j} = 0$ and return to step $(2)$. \\
\end{proposition}

\begin{Rem}
Note that we can shorten step $(5)$ by making the following convention. For any two distinct sets $I,J\subseteq V$ if $I=\emptyset$, then set $\Sigma^{-1}_I=1$ and $\Sigma_{IJ}=0$. This convention will automatically take into account the case when $\pa(j)$ is empty.
\end{Rem}

\begin{proof}
Note that the completion process above starts with the the highest label vertex in $\D$ (this vertex is called a ``source" node and does not have any parents) and the algorithm proceeds in a descending manner. In the process above the positive definiteness condition of the completed matrix will be guaranteed by requiring that each principal submatrix, starting from the highest label and moving down, is positive definite at every step, i.e., positive definiteness is maintained layer by layer. Now suppose that down to an integer $1<\ell\leq p$, the process described above has succeeded in uniquely constructing a positive definite matrix  that corresponds to the principal submatrix $\Sigma_{\{\ell,\ell+1, \ldots, p\}}$ of $\Sigma$. Therefore, the process returns to step $(2)$ with  $j=\ell-1$ and then proceeds to step $(3)$. By this step note that $\Sigma_{\fa(j)}$ is fully determined because, (a) the submatrix $\Sigma_{\pa(j)}$ is specified since it is a submatrix of $\Sigma_{\{\ell,\ell+1, \ldots, p\}}$, where the latter is already determined by the end of the previous step, and (b) the row $\Sigma_{j, \pa(j)}$ is specified since it corresponds to directed edges in the DAG (i.e., $\Gamma\in\mathrm{Q}_{\D}$). Hence $\Sigma_{\fa(j)}$ is fully determined by the beginning of this step. If $\Sigma_{\fa(j)}$ is not positive definite, then the completion in $\mathrm{PD}_{\D}\subseteq\mathrm{PD}_{p}(\R)$ cannot exist as all the principal submatrices of a positive definite matrix also have to be positive definite, i.e., $\Sigma_{\fa(j)}\succ 0$ is a necessary condition to continue with the completion process. We now proceed to show that the condition $\Sigma_{\fa(j)}\succ 0$ is also sufficient for the completion process. In particular, the condition $\Sigma_{\fa(j)}\succ 0$ and Equation \eqref{eq:DL_normal_repeat} as in step $(5)$  uniquely determine the unspecified  entries of the $j$-th column and row of the new submatrix   $\Sigma_{\{j,j+1, \ldots, p\}}$. To see this write
 \[
 \Sigma_{\{j, j+1, \ldots, p\}}=
 \left(
 \begin{matrix}
 \Sigma_{jj}& \Sigma_{j,\pa(j)}&  \Sigma_{j,\pr(j)}\\
  \Sigma_{\pa(j), j}&  \Sigma_{\pa(j)}&  \Sigma_{\pa(j), \pr(j)}\\
   \Sigma_{\pr(j), j} & \Sigma_{\pr(j), \pa(j)}&  \Sigma_{\pr(j)}
 \end{matrix}
 \right).
  \]
  Thus we have
  \begin{align}\label{eq:schur_positive}
 \notag \Sigma_{\{j, j+1, \dots, p\}|\pa(j)}&=\left(\begin{matrix}\Sigma_{jj} &\Sigma_{j, \pr(j)} \\ \Sigma_{\pr(j), j}& \Sigma_{\pr(j)} \end{matrix}\right)-\left(\begin{matrix}\Sigma_{j, \pa(j)} &\Sigma_{\pr(j), \pa(j)}  \end{matrix}\right)(\Sigma_{\pa(j)})^{-1}\left(\begin{matrix}\Sigma_{\pa(j),  j}\\\Sigma_{\pa(j),\pr(j)}\end{matrix}\right)\\
 \notag &=\left(\begin{matrix}\Sigma_{jj|\pa(j)}& \Sigma_{j, \pr(j)}-\Sigma_{j, \pa(j)}(\Sigma_{\pa(j)})^{-1}\Sigma_{\pa(j), \pr(j)}\\
   \Sigma_{\pr(j), j}-\Sigma_{\pr(j), \pa(j)}(\Sigma_{\pa(j)})^{-1}\Sigma_{\pa(j), j} &\Sigma_{\pr(j)|\pa(j)} \end{matrix}\right)\\
   &=\left(\begin{matrix}\Sigma_{jj|\pa(j)}& 0\\
   0 &\Sigma_{\pr(j)|\pa(j)} \end{matrix}\right) \succ 0,
  \end{align}
 
Note that in the second last step we have used the expression 
\[
\Sigma_{\pr(j), j}=\Sigma_{\pr(j), \pa(j)}(\Sigma_{\pa(j)})^{-1}\Sigma_{\pa(j), j}
\]
 from step (5). The positive definiteness in the last step follows  respectively from the facts that, (a) $\Sigma_{jj| \pa(j)}\succ 0$ as this is equivalent to assuming $\Sigma_{\fa(j)}\succ 0$ since $\Sigma_{\pa(j)}\succ 0$, and, (b) $\Sigma_{\pr(j)| \pa(j)}\succ 0$ since it corresponds to a Schur complement of a principal submatrix of the positive definite matrix  $\Sigma_{\{j,j+1,  \ldots, p\}}$.
\end{proof}

 \begin{Rem}
 Note that in Proposition \ref{prop:completion_in_PDD} if we require $\Sigma_{\fa(j)}$ in step $(3)$ to be only invertible, instead of positive definite, then by a similar argument as in the proof of \ref{prop:completion_in_PD}, we can show that, almost everywhere w.r.t Lebesgue measure on $\mathrm{I}_{\D}$, the process in Proposition \ref{prop:completion_in_PDD} yields a matrix, not necessarily positive definite, $\Sigma\in \R^{p\times p}$ that satisfies Equation \eqref{eq:DL_normal}.
 \end{Rem}
 
We now proceed to illustrate the completion process in Proposition \ref{prop:completion_in_PDD} using two examples. The first example illustrates the completion process symbolically and the second example applies to a $\D$-partial matrix with numerical entries.
\begin{figure}[ht]
\centering
\subfigure[]{
\includegraphics[width=2.8cm]{fig_dC4.eps}
\label{fig:dC4}
}
\hspace{3cm}
\subfigure[]{
\includegraphics[width=2.7cm]{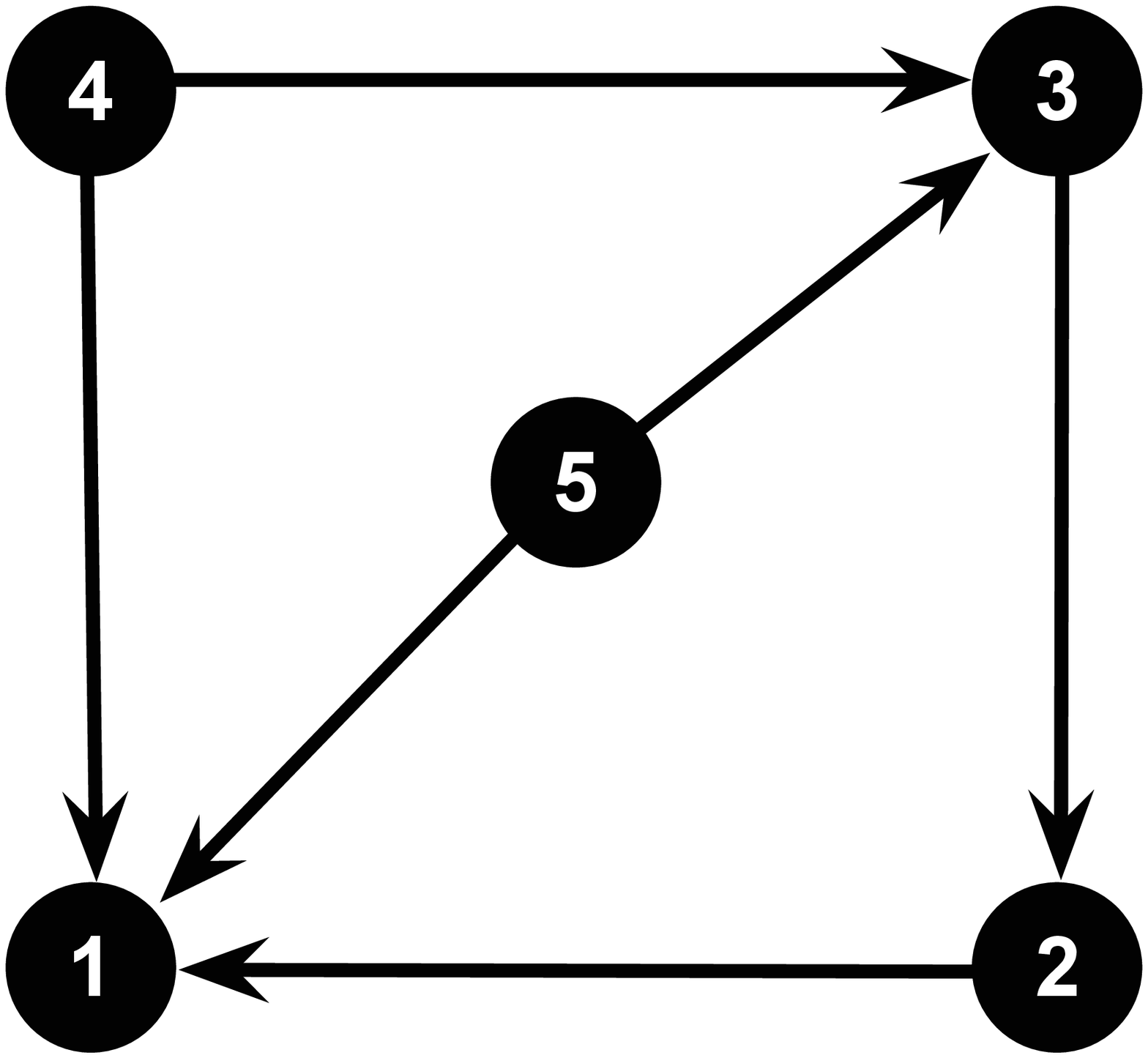}
\label{fig:dg5}
}
\caption{ Completion in  $\mathrm{PD}_\D$ from  Example \ref{ex:completion_on_PDD}. }
\end{figure}

\begin{Ex}\label{ex:completion_on_PDD}
 $(a)$~Consider the DAG $\D$  given in Figure \ref{fig:dC4}. A partial matrix corresponding to $\D$ can be written symbolically as
\[
\Gamma=\left(
\begin{matrix}
\Gamma_{11}&\Gamma_{12}&\Gamma_{13}&*\\
\Gamma_{21}&\Gamma_{22}&* &\Gamma_{24}\\
\Gamma_{31}&* &\Gamma_{33}&\Gamma_{34}\\
 *&\Gamma_{42}&\Gamma_{43}&\Gamma_{44}
\end{matrix}
\right),
\]
where incomplete entries in  $\Gamma$  are denoted by $*$. We now proceed in layers using the steps in Proposition \ref{prop:completion_in_PDD} as $j$ decreases from $4$ to $1$.\\ 

\noindent \underline{Layer: j=4}\\
In step  $(1)$  of the procedure described in Proposition \ref{prop:completion_in_PDD} we have
\[
\Sigma=\left(
\begin{matrix}
\Sigma_{11}&\Sigma_{12}&\Sigma_{13}&?\\
\Sigma_{21}&\Sigma_{22}&? &\Sigma_{24}\\
\Sigma_{31}&?&\Sigma_{33}&\Sigma_{34}\\
 ?&\Sigma_{42}&\Sigma_{43}&\Sigma_{44}
\end{matrix}
\right).
\]

\noindent \underline{Layer: j=3}\\
\noindent In step $(2)$ let $j=4-1=3$. In step $(3)$  either $\Sigma_{\fa(3)}=\left(\begin{matrix}\Sigma_{33}&\Sigma_{34}\\\Sigma_{43}&\Sigma_{44}  \end{matrix}\right)\succ 0$, otherwise the completion in $\mathrm{PD}_{\D}$ does not exist.  Assuming the former, we proceed to step  $(5)$.  Since $\pr(3)=\emptyset$, the layer down to $j=3$ is thus completed.\\

\noindent \underline{Layer: j=2}\\
\noindent We  now return to step  $(2)$  with  $j=3-1=2$. In step $(3)$ we check whether  $\Sigma_{\fa(2)}=\left(\begin{matrix}\Sigma_{22}&\Sigma_{24}\\\Sigma_{42}&\Sigma_{44}  \end{matrix}\right)\succ 0$. Assuming $\Sigma_{\fa(2)}\succ 0$, then in step  $(5)$, as $\pr(2)=\{3\}$, we set  $\Sigma_{32}=\Sigma_{34}\Sigma_{44}^{-1}\Sigma_{42}$ and the layer down to $j=2$ is thus completed. \\

\noindent \underline{Layer: j=1}\\
\noindent Now the process is returned to step $(2)$ with $j=2-1=1$. In step $(3)$ we first check whether
\[
\Sigma_{\fa(1)}=
\left(\begin{matrix}\Sigma_{11}&\Sigma_{12}&\Sigma_{13}\\ \Sigma_{21}&\Sigma_{22}&\Sigma_{34}\Sigma_{44}^{-1}\Sigma_{42}\\ \Sigma_{31}&\Sigma_{34}\Sigma_{44}^{-1}\Sigma_{42}&\Sigma_{33}  \end{matrix}\right)\succ 0.
\]
Assuming $\Sigma_{\fa(1)}\succ 0$, then in step  $(5)$, as $\pr(1)=\{4\}$  we set
\[
\Sigma_{41}=(\Sigma_{42},\Sigma_{43})\left(\begin{matrix}\Sigma_{22}&\Sigma_{34}\Sigma_{44}^{-1}\Sigma_{42}\\\Sigma_{34}\Sigma_{44}^{-1}\Sigma_{42}&\Sigma_{33}  \end{matrix}\right)^{-1}\left(\begin{matrix} \Sigma_{21}\\ \Sigma_{31} \end{matrix}\right).
\]
Now all the unspecified entries are determined and the completed matrix $\Sigma $ is said to be the completion of $\Gamma$ in $\mathrm{PD}_{\D}$.\\

~$(b)$~ Consider the DAG $\D$ given in Figure \ref{fig:dg5}  and  let
\[
\Gamma=\left(
\begin{matrix}
1&0.3&*&0.4&0.6\\
0.3&1&0.4&*&*\\
*&0.4&1&-0.3&0.5\\
0.4&*&-0.3&1&*\\
0.6&*&0.5&*&1
\end{matrix}
\right).
\]
Now by applying the procedure in Proposition \ref{prop:completion_in_PDD}, we start with  $j=5-1=4$. First note that $\Sigma_{\fa(4)}=\Sigma_{44}\succ 0$. In step  $(5)$  we  set  $\Sigma_{45}=0$, since  $\pa(4)=\emptyset$. In the next layer we have $j=3$ and it can be easily verified that 
\[
\Sigma_{\fa(3)}=\left(\begin{matrix}1&-0.3&0.5\\-0.3& 1&0\\ 0.5&0&1 \end{matrix}\right)\succ 0.
\]
Now $\pr(3)$ is empty so by step $(4)$ we return to step $(2)$ since $(5)$ in the completion process is redundant. Now in step $(3)$ for  $j=2$ it is obvious that
\[
\Sigma_{\fa(2)}=
\left(
\begin{matrix}
1& 0.4\\
0.4&1
\end{matrix}
\right)
\succ 0.
\]
Therefore we proceed to step  $(5)$ and calculate $\Sigma_{\pr(2), 2}=(\Sigma_{42},\Sigma_{52})=(-0.12, 0.2)$. Moving to $j=1$, it is easily verified that 
\[
\Sigma_{\fa(1)}=\left(\begin{matrix}1&0.3&0.4&0.6\\0.3&1&-0.12&0.2\\0.4&-0.12& 1&0\\0.6& 0.2&0&1 \end{matrix}\right)\succ 0.
\]
Finally, we proceed to step (4), where we have $\Sigma_{\pr(1), 1}=\Sigma_{31}$ and we set
\[
\Sigma_{31}=(0.4, -0.3, 0.5)\left(\begin{matrix} 1&-0.12&0.2\\-0.12&1&0\\0.2&0&1 \end{matrix}\right)^{-1}\left(\begin{matrix} 0.3\\0.4\\0.6 \end{matrix}\right)=0.2437.
\]

The process now terminates and the unique completion of $\Gamma$  in $\mathrm{PD}_{\D}$  is given by
\[
\Sigma=\left(
\begin{matrix}
1&0.3&0.2437&0.4&0.6\\
0.3&1&0.4&-0.12&0.2\\
0.2437&0.4&1&-0.3&0.5\\
0.4&-0.12&-0.3&1&0\\
0.6&0.2&0.5&0&1
\end{matrix}
\right).
\]
\noindent
In order to double check that $\Sigma$ is indeed in $\mathrm{PD}_{\D}$ we first compute $\Sigma^{-1}$:
\[
\Omega =\Sigma^{-1}=
\left(
\begin{matrix}
    2.353 &  -0.567  & 0  & -1.009  & -1.298\\
   -0.567  &  1.327  & -0.476 &   0.243 &   0.313\\
    0 &  -0.476  &  1.706  &  0.455 &  -0.758\\
   -1.009  &  0.243   & 0.455 &   1.569  &  0.330\\
   -1.298  &  0.313 &  -0.758   & 0.330    &2.095
   \end{matrix}
\right).
\]
Now the lower triangular matrix $R$ in the standard Cholesky decomposition of $\Omega$ is given by
\[
R=
\left(
\begin{matrix}
    1.534    &     0  &       0   &      0      &   0\\
   -0.370   & 1.091      &   0       &  0    &     0\\
       0     &  -0.436  &  1.231   &      0     &    0\\
   -0.658 &   0    &  0.369  &  1     &    0\\
   -0.846   & 0 &  -0.616 &   0   & 1
\end{matrix}
\right),
\]
which clearly shows that the lower triangular matrix $L$ in the corresponding modified Cholesky decomposition of $\Omega$ is in $\mathcal{L}_{\D}$.

\end{Ex}
Proposition \ref{prop:completion_in_PDD} establishes conditions under which a $\D$-partial matrix can be completed in $\mathrm{PD}_{\D}$. It therefore establishes conditions for the existence and uniqueness of the completion. A natural question to ask is if there are simple conditions which guarantee this completion. We now deduce from Proposition \ref{prop:completion_in_PDD} that if the digraph $\D$ is a perfect DAG then $\Gamma\in \mathrm{Q}_{\D}$ can always be completed in $\mathrm{PD}_{\D}$.

\begin{corollary}\label{cor:completion_per}
Let $\D$ be a perfect DAG and $\Gamma\in \mathrm{I}_\D$. Then  a necessary and sufficient condition for completing $\Gamma$ in $\mathrm{PD}_{\D}$ is that $\Gamma\in \mathrm{Q}_{\D}$. Moreover, if $\Gamma\in \mathrm{Q}_{\D}$, then $\Gamma$  can be simply completed to $\Sigma\in \mathrm{PD}_{\D}$ as follows:\\
~$(a)$~ Set $\Sigma_{ij}=\Gamma_{ij}$\;\: for each $\{i,j\}\in\mathscr{V}$,\\
~$(b)$~ Set $\Sigma_{\pr(j), j}=\Sigma_{\pr(j), \pa(j)}\Sigma_{\pa(j)}^{-1}\Sigma_{\pa(j),  j}$ \: and  $\Sigma_{j, \pr(j)}=\Sigma_{\pr(j), j}'$\:\: for each  $j=p-1, \ldots, 1$.
\end{corollary}

\begin{proof}
($\Rightarrow$) 
If $\Gamma\notin \mathrm{Q}_{\D}$, then  $\Gamma$ cannot be completed to a positive definite matrix. This follows easily from the fact that the principal minors of a positive definite matrix are all strictly positive. In particular, if a completion for $\Gamma$ in $\mathrm{PD}_{\D}$ exists, then $\Gamma_{C}\succ 0$ for each $C\in \mathscr{C}_{\G}$, which implies that $\Gamma\in \mathrm{Q}_{\D}$.  

($\Leftarrow$) 
Now assume that  $\Gamma\in \mathrm{Q}_{\D}$. As $\D$ is perfect, by definition there are no immoralities present in $\D$, (i.e., all parents of each node are adjacent). Hence $\fa(j)$ is a complete subset of $V$ for each $j$. It is therefore contained in a clique of $\D$, and hence $\Sigma_{\fa(j)}=\Gamma_{\fa(j)}\succ 0 $ since $\Gamma\in \mathrm{Q}_{\D}$. Thus  step $(3)$ of the completion process in Proposition \ref{prop:completion_in_PDD} is always satisfied and can thus be omitted from the procedure. We can therefore conclude that $\Gamma$ can alway be completed in $\mathrm{PD}_{\D}$.

Now since $\Sigma_{\pa(j)}$ are all already available step $(5)$ can be performed for $j=p-1,\ldots,1$. The proof of Proposition \ref{prop:completion_in_PDD} demonstrates that after performing these steps, we obtain the completion of $\Gamma$ in $\mathrm{PD}_{\D}$.
\end{proof}
An alternative but similar procedure to that in Proposition \ref{prop:completion_in_PDD} for completing a $\D$-partial matrix $\Gamma$ in $\mathrm{PD}_{\D}$ is to construct a finite sequence of DAGs, $\D^{(0)}, \dots, \D^{(n)}$ such that $\D^{(n)}$ at the end of this sequence is perfect, and then a partial matrix, $\Gamma^{(n)}$ over $\D^{(n)}$. The first DAG in this sequence is $\D^{(0)}=\D$. If $\D^{(0)}$ is not perfect, then for each immorality of the form $i\rightarrow k\leftarrow j$ in $\D^{(0)}$ with $i>j$ we add a directed edge $i\rightarrow j$ to the edge set of $\D^{(0)}$. Let $\D^{(1)}$ be the DAG with  added edges. It is clear that $\D^{(0)}$ is an induced subgraph of $\D^{(1)}$. We continue this process until we obtain a perfect DAG. Therefore, we have a finite sequence $\D^{(0)}, \dots, \D^{(n)}$ of DAGs such that $\D^{(n)}$ at the end of this sequence is perfect. Now starting from the largest  vertex $j$, we use Equation \eqref{eq:DL_normal_repeat} to compute the entries $\Gamma_{ij}^{(n)}$ that correspond to added  edges $i\rightarrow j$ in $\D^{(n)}$. Unless for some $j$ the requirement  $\Sigma_{\pa(j)}\succ 0$ is not met  in Equation \eqref{eq:DL_normal_repeat}, this process succeeds in filling in those unspecified entries of $\Sigma$ that correspond to $\D^{(n)}$, i.e., we obtain a partial matrix $\Gamma^{(n)}$  over  $\D^{(n)}$. Since $\D^{(n)}$ is a perfect DAG, by Corollary \ref{cor:completion_per}, the partial matrix  $\Gamma^{(n)}$  can be completed in $\mathrm{PD}_{\D^{(n)}}$  if and only if it belongs to $\mathrm{Q}_{\D^{(n)}}$. Furthermore, $\Gamma^{(n)}$  can be completed by following the simple non-recursive completion procedure described in Corollary \ref{cor:completion_per}. It is clear that completion of $\Gamma^{(n)}$ in $\mathrm{PD}_{\D^{(n)}}$ is also the completion of $\Gamma$  in $\mathrm{PD}_{\D}$. 
We illustrate this alternative procedure by an example.
\begin{figure}[ht]
\centering
\subfigure[]{
\includegraphics[width=3cm]{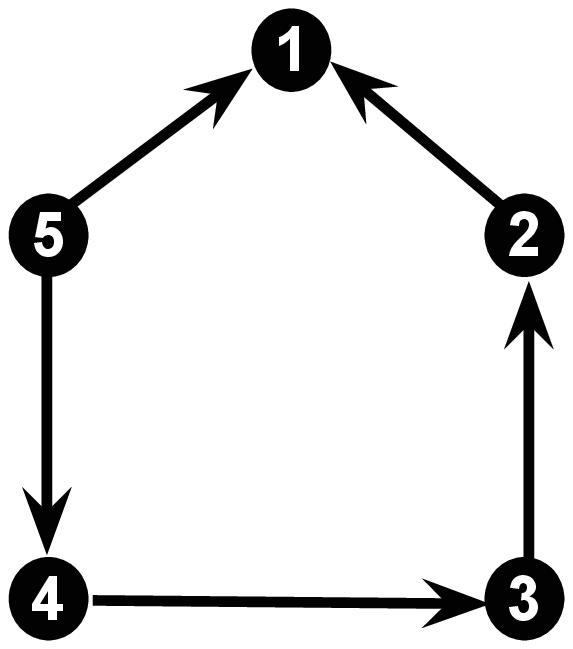}
\label{fig:alt_figa}
}
\hspace{3cm}
\subfigure[]{
\includegraphics[width=3cm]{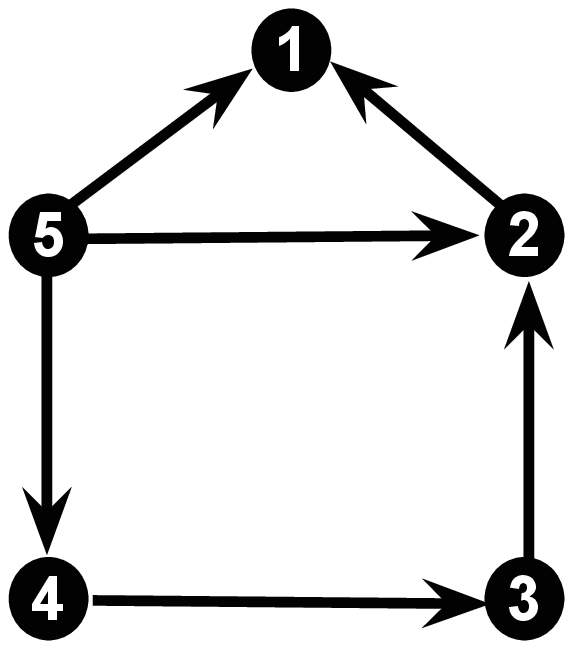}
\label{fig:alt_figb}
}
\hspace{3cm}
\subfigure[]{
\includegraphics[width=3cm]{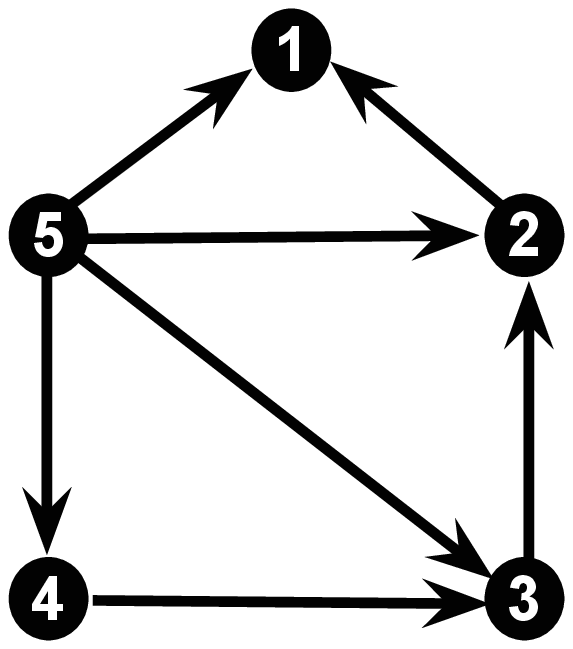}
\label{fig:alt_figc}
}
\hspace{3cm}
\subfigure[]{
\includegraphics[width=3cm]{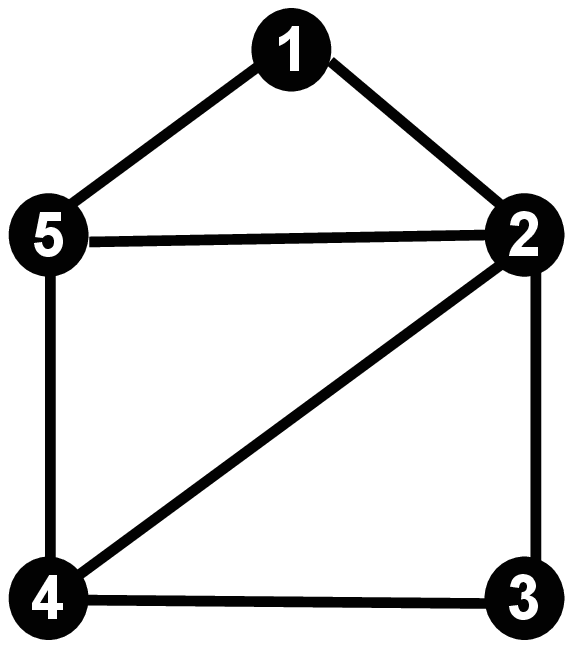}
\label{fig:alt_figd}
}
\caption{A finite sequence $\D_{0},\D_{1},\D_{2}$ of DAGs and the undirected version of $\D_{2}$ . }
\end{figure}

\begin{Ex} Let $\D$ be the DAG in Figure \ref{fig:alt_figa}.  Starting from $\D_{0}=\D$, the only immorality in this DAG is $5\rightarrow 1\leftarrow 2$. By adding the directed edge $5\rightarrow 2$ we obtain $\D_{1}$ in Figure \ref{fig:alt_figb}. Next we obtain the perfect DAG $\D_{2}$ in Figure \ref{fig:alt_figc}, by adding the directed edge $5\rightarrow 3$ corresponding to the immorality $5\rightarrow 2\leftarrow 3$ in $\D_{1}$.  Now consider the completion of the $\D$-partial matrix
\[
\Gamma=
\left(
\begin{matrix}
\Gamma_{11}&\Gamma_{12}&*&*&\Gamma_{15}\\
\Gamma_{21}&\Gamma_{22}&\Gamma_{23}&*&*\\
*&\Gamma_{32}&\Gamma_{33}&\Gamma_{34}&*\\
*&*&\Gamma_{43}&\Gamma_{44}&\Gamma_{45}\\
\Gamma_{15}&*&*&\Gamma_{54}&\Gamma_{55}
\end{matrix}
\right).
\] 
From  Equation \ref{eq:DL_normal_repeat} we compute 
\[
\Gamma_{53}=\Gamma_{54}\Gamma_{44}^{-1}\Gamma_{43},\:\:\text{and}\:\: \Gamma_{52}=\Gamma_{53}\Gamma_{33}^{-1}\Gamma_{32}=\Gamma_{54}\Gamma_{44}^{-1}\Gamma_{43}\Gamma_{33}^{-1}\Gamma_{32}.
\] 
Thus we obtain the following partial matrix over the perfect DAG $\D_{2}$ (or over the decomposable graph in Figure \ref{fig:alt_figd}):
\[
\Gamma^{(2)}=
\left(
\begin{matrix}
\Gamma_{11}&\Gamma_{12}&*&*&\Gamma_{15}\\
\Gamma_{21}&\Gamma_{22}&\Gamma_{23}&*&\Gamma_{54}\Gamma_{44}^{-1}\Gamma_{43}\\
*&\Gamma_{32}&\Gamma_{33}&\Gamma_{34}&\Gamma_{53}\Gamma_{33}^{-1}\Gamma_{32}\\
*&*&\Gamma_{43}&\Gamma_{44}&\Gamma_{45}\\
\Gamma_{15}&\Gamma_{54}\Gamma_{44}^{-1}\Gamma_{43}&\Gamma_{53}\Gamma_{33}^{-1}\Gamma_{32}&\Gamma_{54}&\Gamma_{55}
\end{matrix}
\right).
\]
\end{Ex}
\section{Completable DAGs and generalization of Grone et al.\cite{Grone1984}'s result}\label{sec:gen}
 A pertinent question in the positive definite completion problem for DAGs is the class of DAGs for which the  completion of a partial matrix  in $\mathrm{PD}_{\D}$ is certain to exist. Corollary \ref{cor:completion_per} asserts that if $\D$ is perfect and the $\D$-partial matrix $\Gamma\in \mathrm{Q}_{\D}$ then a completion in $\mathrm{PD}_\D$ always exists. Is the class of perfect graphs maximal in the sense that only for this class of graphs is completion guaranteed for all $\Gamma\in \mathrm{Q}_{\D}$ ? It is evident that a necessary condition for the existence of the completion in  $\mathrm{PD}_\D$, or any positive definite completion for that matter, is  that  $\Gamma$ is a partial positive definite matrix over  $\D$, i.e.,  $\Gamma\in \mathrm{Q}_{\D}$.

 \begin{theorem}\label{thm:existence_completion}
 Every partial positive definite matrix  over  $\D$  can be completed in  $\mathrm{PD}_{\D}$   if and only if  $\D$  is a perfect  DAG.
  \end{theorem}
   \begin{proof}
We proceed using a proof by contradiction embedded in an induction argument. Assume the statement of the theorem is true for any DAG s.t. $ |V| < p$. We shall prove  the theorem for $|V|=p$. The case $p=1$ is trivially true, hence  let  $p\geq 2$. Let  $\D_{[1]}$ denote the induced DAG on  $V\setminus\{1\}$.\\

\noindent
~$\Longrightarrow )$ \quad  Suppose that every partial positive definite matrix over  ${\D}$ can be completed in $\mathrm{PD}_{\D}$. Let $\Gamma^{[1]}$  be an arbitrary element in $\mathrm{Q}_{\D_{[1]}}$.  Now let us define the $\D$- partial matrix $\Gamma$  such that for each $\{i, j\}\in \mathscr{V}$ 
\[
\Gamma_{ij}=
\begin{cases}
1& \text{if $i=1, j=1$,}\\
0&\text{if $i=1,\; j\neq 1$,}\\
\Gamma_{ij}^{[1]}& \text{otherwise.}
\end{cases}
\]  
It is clear that $\Gamma$ is a partial positive definite matrix in $\mathrm{Q}_{\D}$. By assumption, $\Gamma$ can be completed to a positive definite matrix $\Sigma$ in  $\mathrm{PD}_{\D}$. Now note that the principal submatrix $\Sigma_{V\setminus\{1\}}$ is positive definite and satisfies Equation \eqref{eq:DL_normal_repeat} w.r.t.  $\D_{[1]}$. In particular, $\Sigma_{V\setminus\{1\}}$ is the  completion of  $\Gamma^{[1]}$  in   $\mathrm{PD}_{\D_{[1]}}$. By the induction hypothesis this implies that $\D_{[1]}$  is a perfect DAG.  Therefore, to show that $\D$  is perfect, it suffices to show that  $\mathrm{pa}(1)$ is  a complete subset of $V$. On the contrary, if $\mathrm{pa}(1)$ is not complete, then there are non-adjacent vertices $i_{1},j_{1}\in V$   such that  $i_{1}\rightarrow 1\leftarrow j_{1}$. Assume w.l.o.g that  $i_{1}>j_{1}$.  In particular, this implies that $i_{1}$ is a predecessor of $j_{1}$, i.e., $ i_{1}\in \pr(j_{1})$.  Fix an arbitrary number  $\epsilon$ in the open interval $(\sqrt{2}/2, 1)$.  Consider the $\D$-partial matrix  $\Gamma$ that is defined  for each  $\{ i, j\}\in \mathscr{V}$ as
\[
\Gamma_{ij}=
\begin{cases}
1&\text{if \: $i=j$,}\\
\epsilon&\text{if \: $i=1, \; j=j_{1}$ \: or \: $i=i_{1}, \; j=1$,}\\
0& \text{otherwise}
\end{cases}
 \]
 One can  easily check that $\Gamma \in \mathrm{Q}_{\D}$. Let  $\Sigma$  denote the  completion of  $\Gamma$  in  $\mathrm{PD}_{\D}$.  Since $\Sigma_{\pa(j_{1}),  j_{1}}=\Gamma_{\pa(j_{1}), j_{1}}=0$ from Equation \eqref{eq:DL_normal_repeat} we have
 \[
  \Sigma_{\pr(j_{1}), j_{1}}=\Sigma_{\pr(j_{1}),  \pa(j_{1})}(\Sigma_{\pa( j_{1})})^{-1}\Sigma_{\pa(j_{1}), j_{1}}=0. 
  \]
  In particular, $\Sigma_{i_{i}j_{1}}=0$. Therefore
 \[
 \Sigma_{\{1, i_{1}, j_{1}\}}=
 \left(
 \begin{matrix}
 1& \epsilon&\epsilon\\
 \epsilon&1&0\\
 \epsilon& 0& 1
 \end{matrix}
 \right)\succ 0.
 \]
 However, this matrix is positive definite if and only if $\epsilon\in (0,\sqrt{2}/2)$, yielding  a contradiction.\\
 
 $\Longleftarrow)$\quad Let  $\D$ be a perfect DAG and  $\Gamma\in \mathrm{Q}_\D$. Let   $\Gamma^{[1]}$ be  the restriction of $\Gamma$ to  $V\setminus\{1\}$  and  $\Sigma^{[1]}$   the completion of  $\Gamma^{[1]}$  in  $\mathrm{PD}_{\D_{[1]}}$. Now define  $\Sigma$ as follows:  
 \[
\Sigma_{V\setminus\{1\}}=\Sigma^{[1]},\:  \Sigma_{\fa(1), 1}=\Gamma_{\fa(1), 1} \:\; \text{and} \:\; \Sigma_{\pr(1), 1}=\Sigma_{\pr(1),  \pa(1)}(\Sigma_{\pa(1)})^{-1}\Sigma_{\pa(1), 1}.
 \]
 Then using Equation \eqref{eq:schur_positive} when $j=1$  shows that  $\Sigma\succ 0$.
 \end{proof}
 \begin{Rem}
Notice that the $\Longleftarrow)$\: part of Theorem \ref{thm:existence_completion} was also proved in Corollary \ref{cor:completion_per}.
\end{Rem}

 \begin{Rem}
 In the context of undirected graphs,  Grone et al. \cite{Grone1984} prove that every partial positive definite matrix can be completed to a positive definite matrix if and only if the underlying graph is decomposable. Theorem \ref{thm:existence_completion} above is the corresponding result in the DAG context with the caveat that the result in \cite{Grone1984} for undirected graphs does not imply the result below. In particular, \cite{Grone1984} implies that $\D^u$ has to be decomposable if a positive definite completion is to be guaranteed for any arbitrary partial positive definite matrix. The requirement that the graph $\D^u$ be decomposable does not however mean that $\D$ is perfect.  
 
 \end{Rem}

\begin{corollary}\label{cor:completion_UG}
Suppose $\G$ is a decomposable graph. Then every $\Gamma\in \mathrm{Q}_{\G}$ can be completed to a unique $\Sigma$ in $\mathrm{PD}_{\G}$. Consequently, every partial positive definite matrix over a decomposable graph has a positive definite completion.
 \end{corollary}
 \begin{proof}
Let $\Gamma$  be in $\mathrm{Q}_{\G}$ and let $\D$ be a perfect DAG version of $\G$. Then $\Gamma\in\mathrm{Q}_{\D}$. The result is immediate by applying Corollary \ref{cor:completion_per} followed by Corollary \ref{cor:per_dec}.
  \end{proof}

There is an interesting contrast  between completing a  given partial positive definite matrix  $\Gamma\in \mathrm{Q}_{D}$ in $\mathrm{PD}_{\G}$ vs. completing it in $\mathrm{PD}_{\D}$. In particular, Theorem 3. in \cite{Grone1984} asserts that $\Gamma\in \mathrm{Q}_{\G}$ can be completed in $\mathrm{PD}_{\G}$ if \emph{any} positive completion exists. A completion in $\mathrm{PD}_{\D}$ is therefore sufficient to guarantee a completion in $\mathrm{PD}_{\G}$. The other way around is not true. In particular, $\Gamma$ may not be completed in $\mathrm{PD}_\D$ even when it can be completed in $\mathrm{PD}_{\G}$. This is because completion in $\mathrm{PD}_{\D}$ is more restrictive than completion in $\mathrm{PD}_{\G}$. We illustrate this distinction in the following example.

 \begin{figure}[htbp]
\begin{center}
\includegraphics[width=2.8cm]{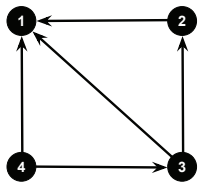}
\caption{A non-perfect DAG from Example \ref{ex:no_completion}}
\label{fig:nonpdag4}
\end{center}
\end{figure}

 \begin{Ex}\label{ex:no_completion}
 Consider the following partial positive definite matrix over the DAG in Figure \ref{fig:nonpdag4}.
 \[
 \Gamma=
 \left(
\begin{matrix}
7 &  12 &  12 &  16\\
  12  & 30 &  28 &  *\\
  12  & 28 &  37 &  32\\
 16 &  * &  32 &  38
\end{matrix}
\right).
\]
Although $\D$  is not a perfect DAG  we have $\G$, the undirected version of  $\D$,  is decomposable and therefore by Corollary \ref{cor:completion_UG} it can be completed to a positive definite matrix  in  $\mathrm{PD}_{\G}$. However, completion of   $\Gamma$  in  $\mathrm{PD}_{\D}$  requires that  $\Sigma_{42}=\Gamma_{43}\Gamma_{33}^{-1}\Gamma_{32}=24.2162$ and one can check that the completed matrix
\[
\left(
\begin{matrix}
7 &  12 &  12 &  16\\
  12  & 30 &  28 &  24.2162\\
  12  & 28 &  37 &  32\\
 16 &  24.2162 &  32 &  38
\end{matrix}
\right)
\]
is not positive definite.  Consequently, $\Gamma$  cannot be completed in $\mathrm{PD}_{\D}$.
 \end{Ex}
 Now suppose that $\G$ is an undirected graph and $\Gamma \in \mathrm{Q}_{\G}$. When $\G$ is decomposable it has a DAG version $\D$ that is perfect. Corollary \ref{cor:completion_per} and the fact that $\mathrm{PD}_{\D} = \mathrm{PD}_{\G}$ (see Proposition \ref{prop:equiv}) imply that $\Gamma$ can be explicitly completed in $\mathrm{PD}_{G}$. However, when $\G$ is not decomposable no algorithm for completing $\Gamma$ in $\mathrm{PD}_{\G}$ is known, unless a positive definite completion in $\mathrm{PD}_{p}(\R)$ is given or guaranteed. Given $\G$, in practice, it is useful to know whether $\Gamma$ can be completed in \emph{some} DAG version of this undirected graph $\G = \D^u$. This is because if $\Gamma$ can be completed w.r.t a DAG version of $\G$, then the completion process given in Proposition \ref{prop:completion_in_PDD} can be exploited to obtain a positive definite completion, and thus ensuring the existence of a completion in $\mathrm{PD}_{\G}$. In the following example we show that even when $\Gamma$ can be completed in $\mathrm{PD}_{\G}$ the completion may still not exist for any DAG version of $\G$.

\begin{figure}[htbp]
\begin{center}
\includegraphics[width=3.5cm]{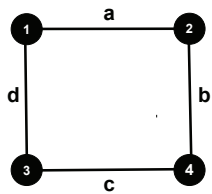}
\caption{Positive definite completion for  $C_{4} $ from Example \ref{cycle_UG_DAG} }
\label{fig:aC4}
\end{center}
\end{figure}
 \begin{Ex} \label{cycle_UG_DAG}
 Consider the partial matrix
 \[
 \Gamma=
 \left(
\begin{matrix}
1&a&d&*\\
a&1&*&b\\
d&*&1&c\\
*&b&c&1
\end{matrix}
\right)
 \]
 over the four cycle  $C_{4}$ as given in Figure \ref{fig:aC4}. It is clear that when $|a|,|b|,|c|,|d|<1$, $\Gamma$ is a partial positive definite matrix over $C_{4}$.  It is shown in \cite{Barrett1993} that $\Gamma$ can be completed to a positive definite matrix  $\Sigma$  if and only if 
 \[
 f(a,b,c,d)=\sqrt{(1-a^{2})(1-b^{2})}+\sqrt{(1-c^{2})(1-d^{2})}-|ab-cd|>0.
 \]
 Now consider the list of the DAG versions of $C_{4}$ as given in Table \ref{table:dC4}. A simple enumeration will demonstrate that the list in Table \ref{table:dC4} is exhaustive and contains all DAG version of  $C_{4}$.

 \end{Ex}
 
\begin{table}[htbp]
\caption{DAG versions of  $C_{4}$ from Example \ref{cycle_UG_DAG}.}
\begin{center}
\begin{tabular}{cccc}
\includegraphics[width=3.1cm]{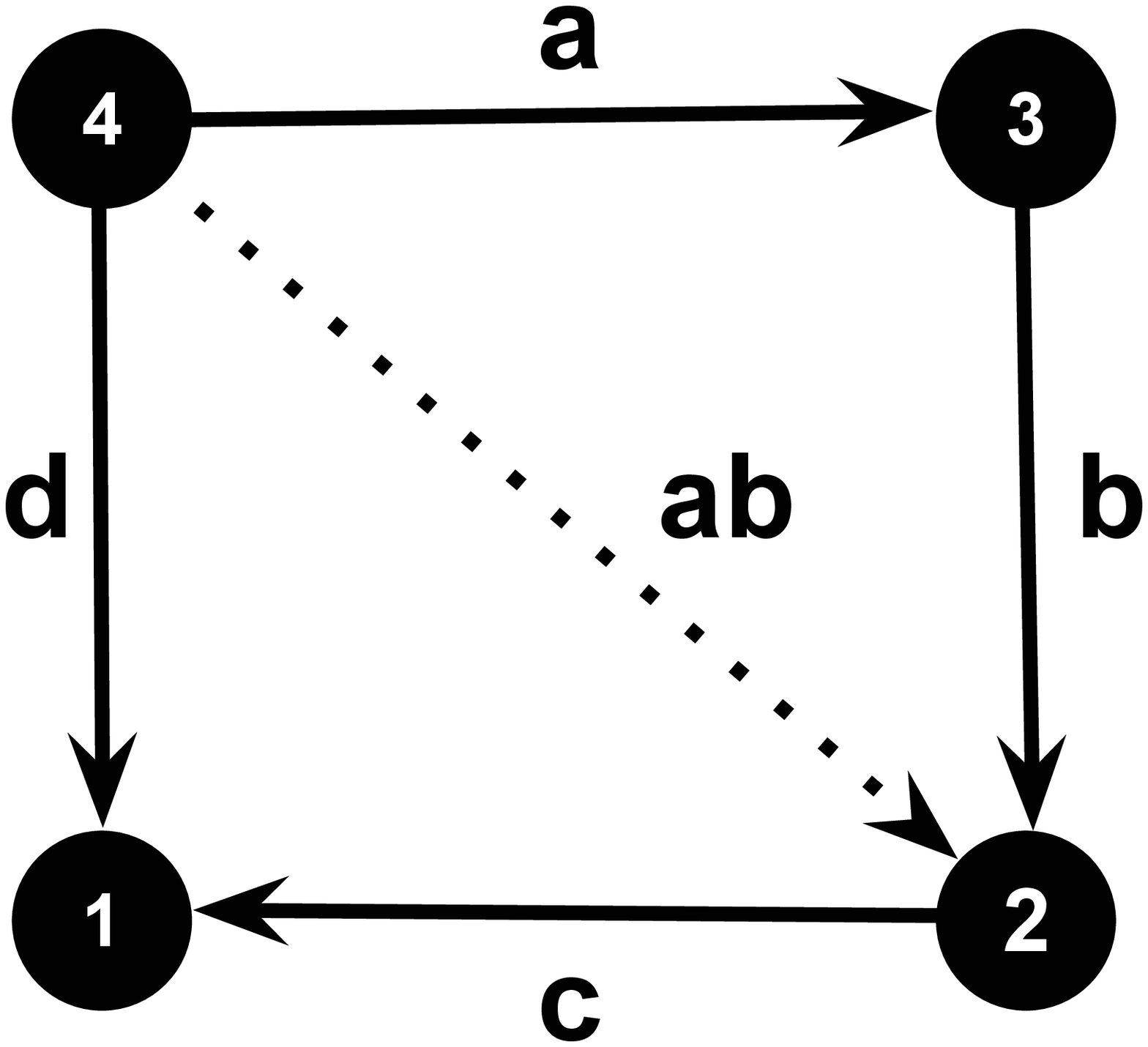}&\includegraphics[width=3.1cm]{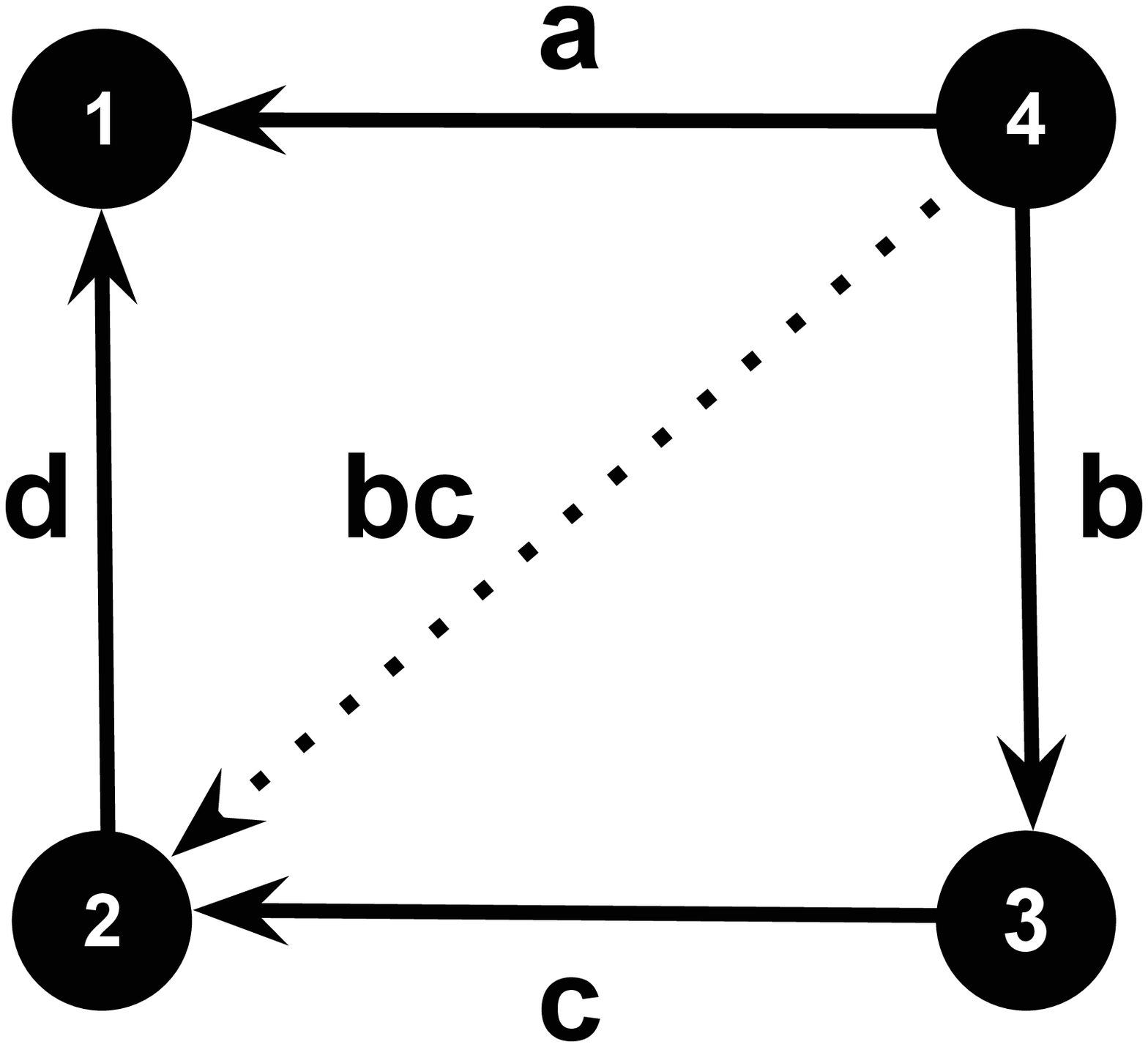}&\includegraphics[width=3.1cm]{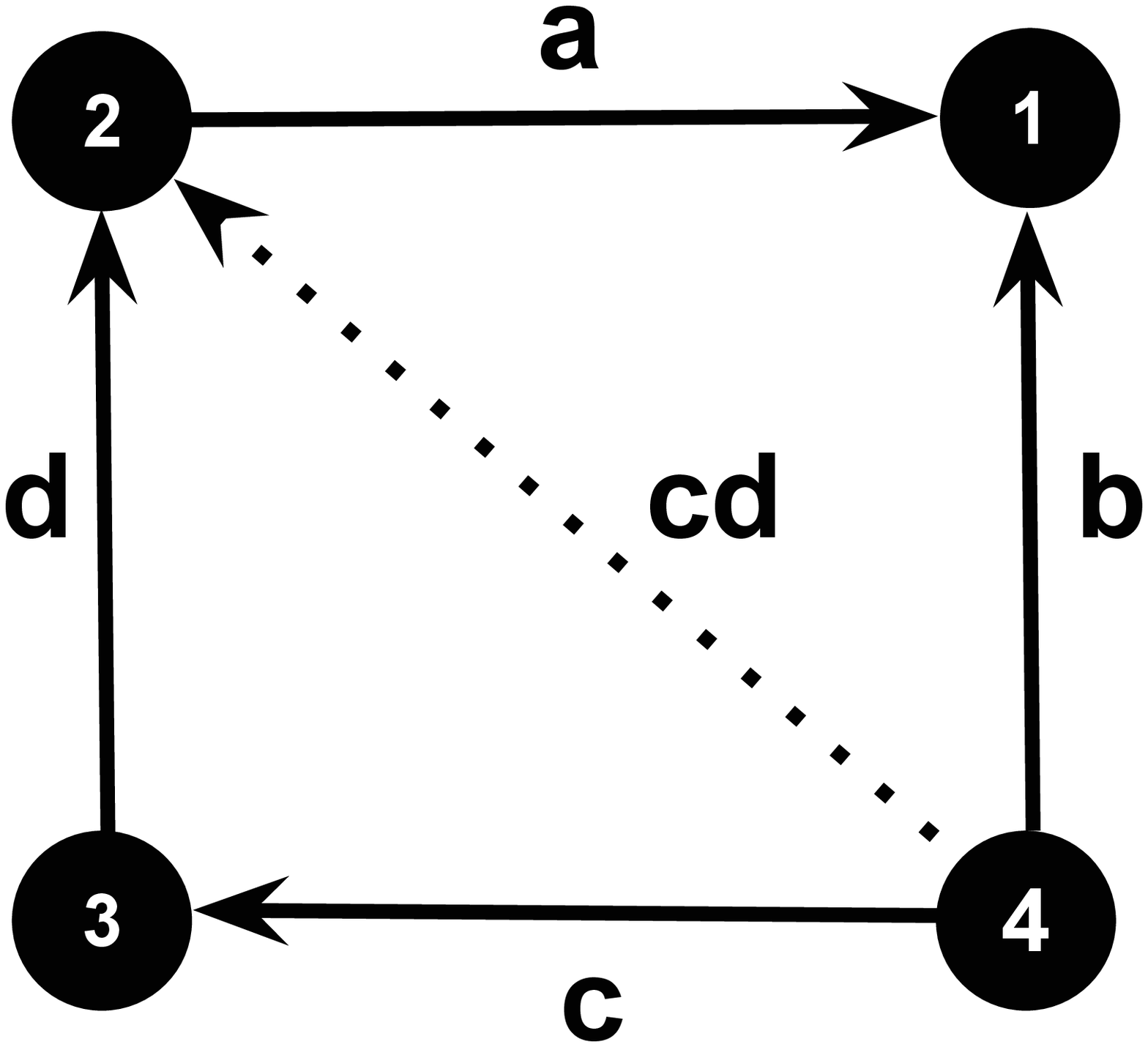}&\includegraphics[width=3.1cm]{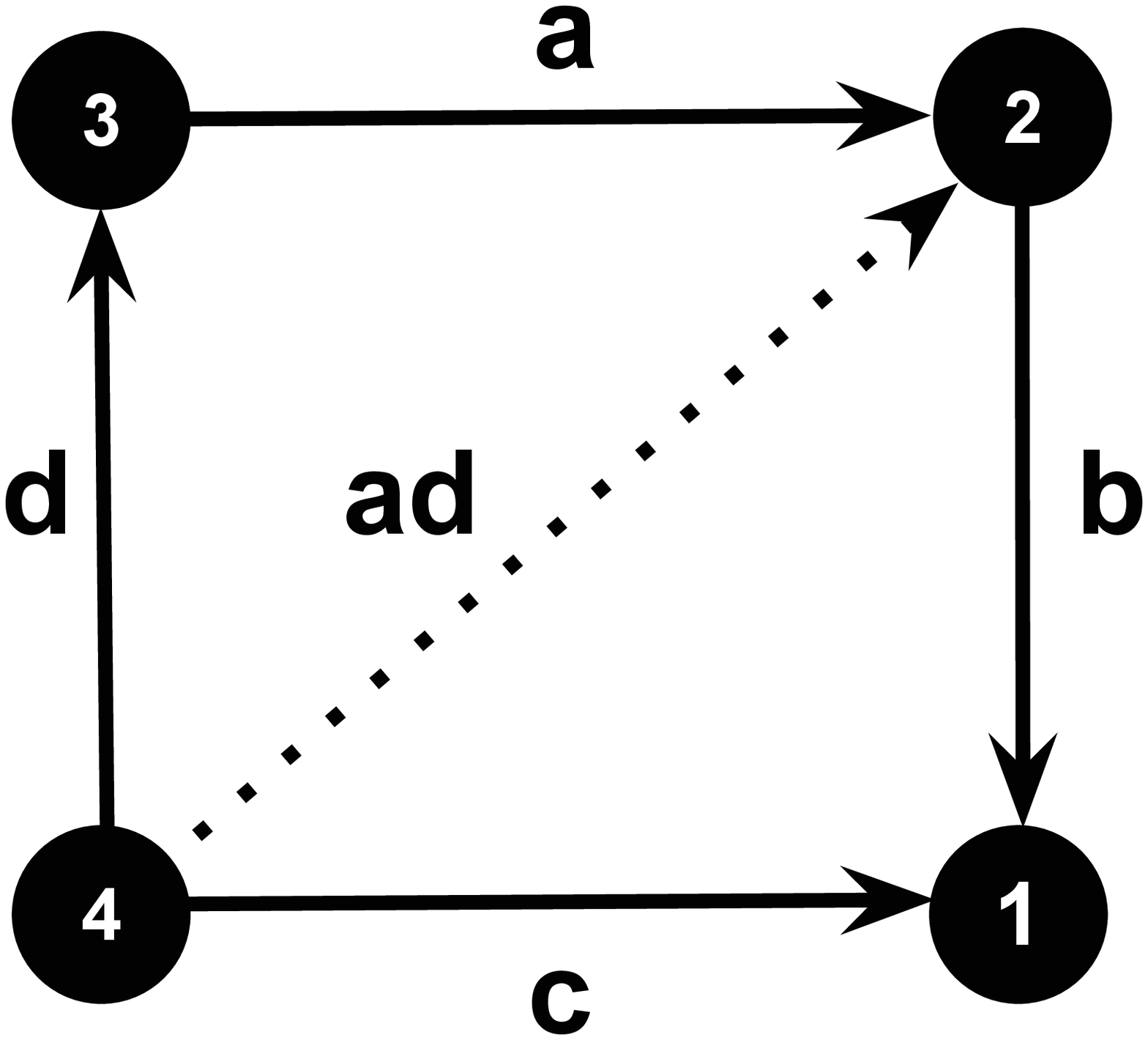}\\
(1)&(2)&(3)&(4)\\
& & &\\
\includegraphics[width=3.1cm]{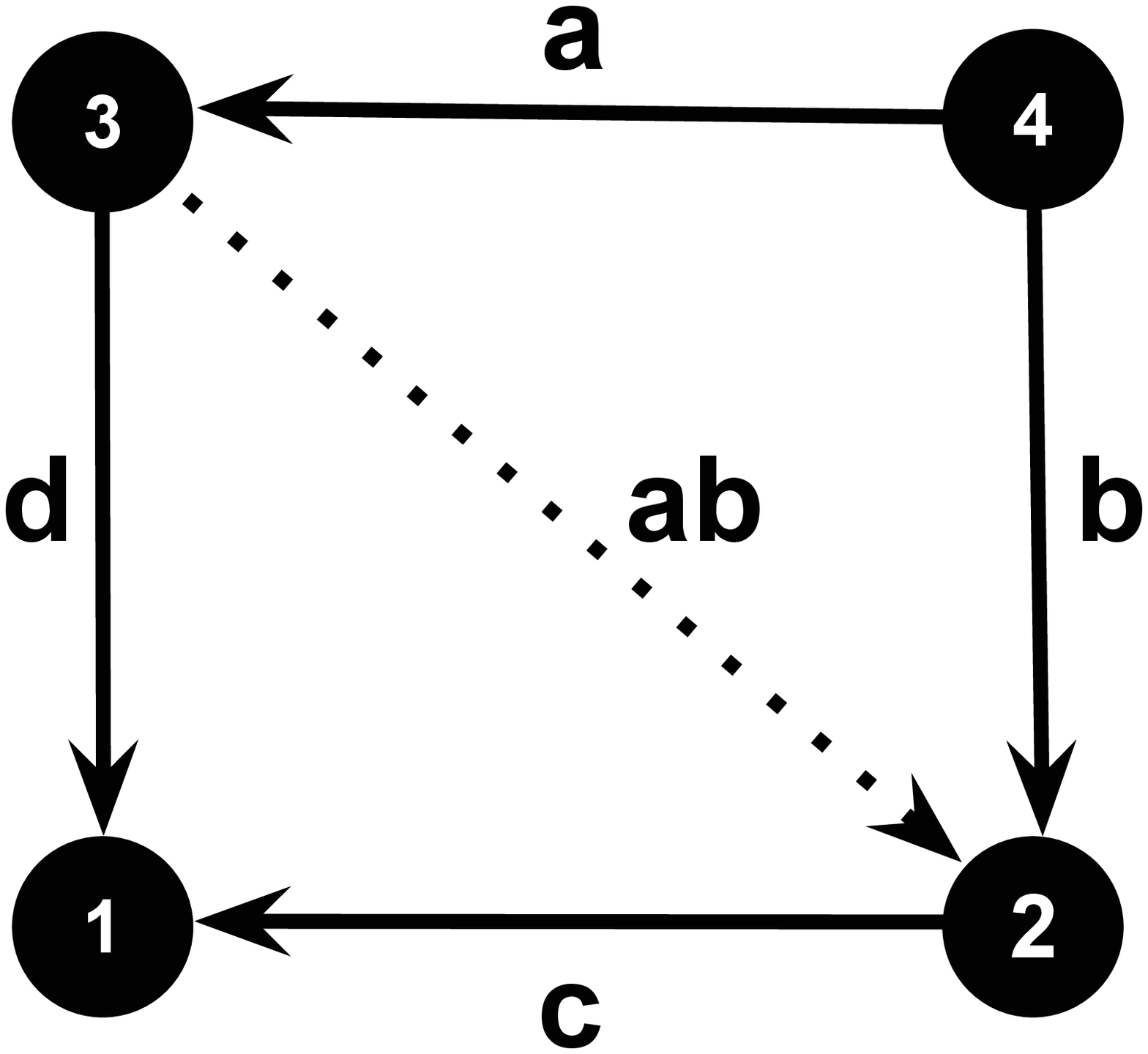}&\includegraphics[width=3.1cm]{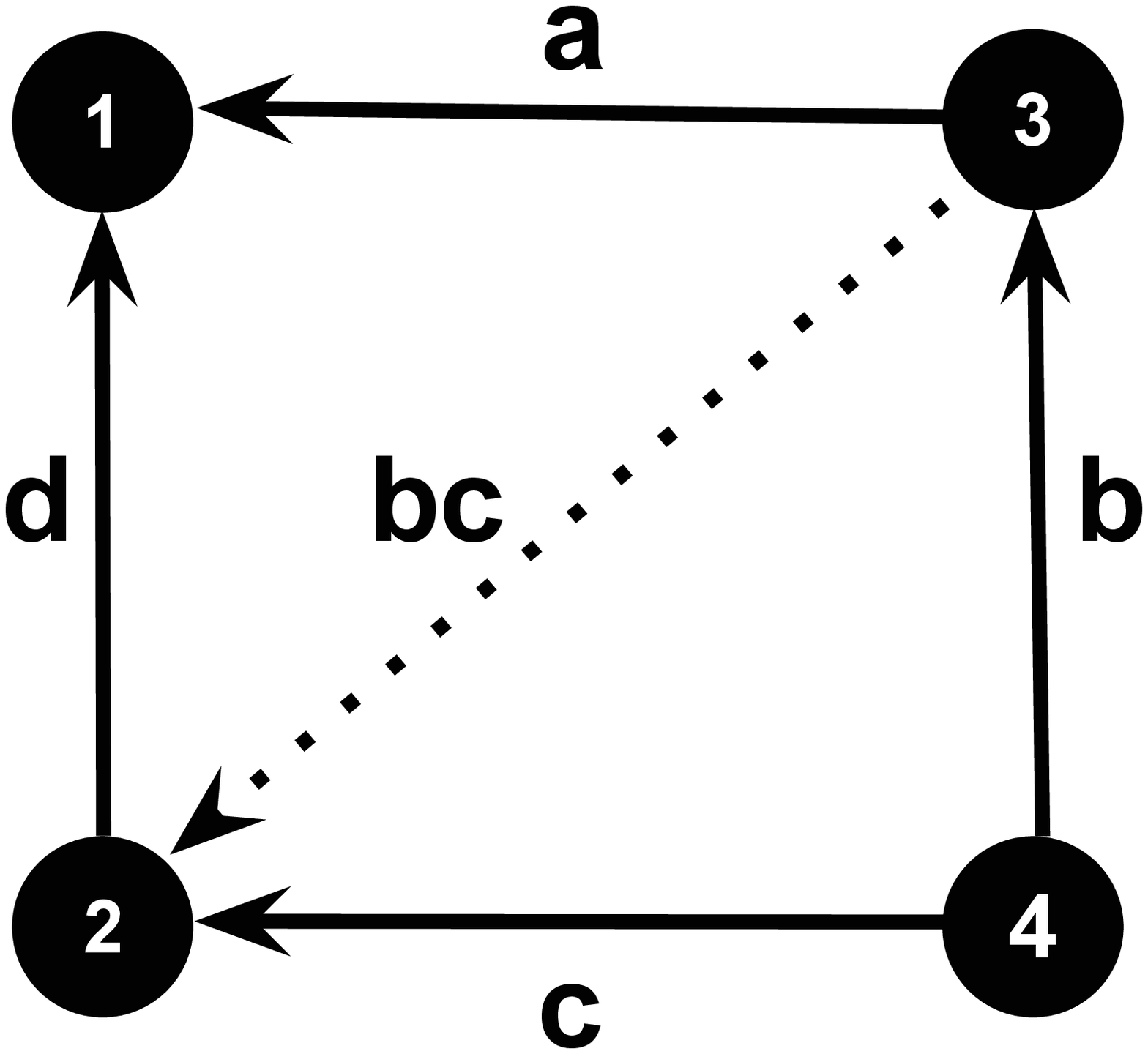}&\includegraphics[width=3.1cm]{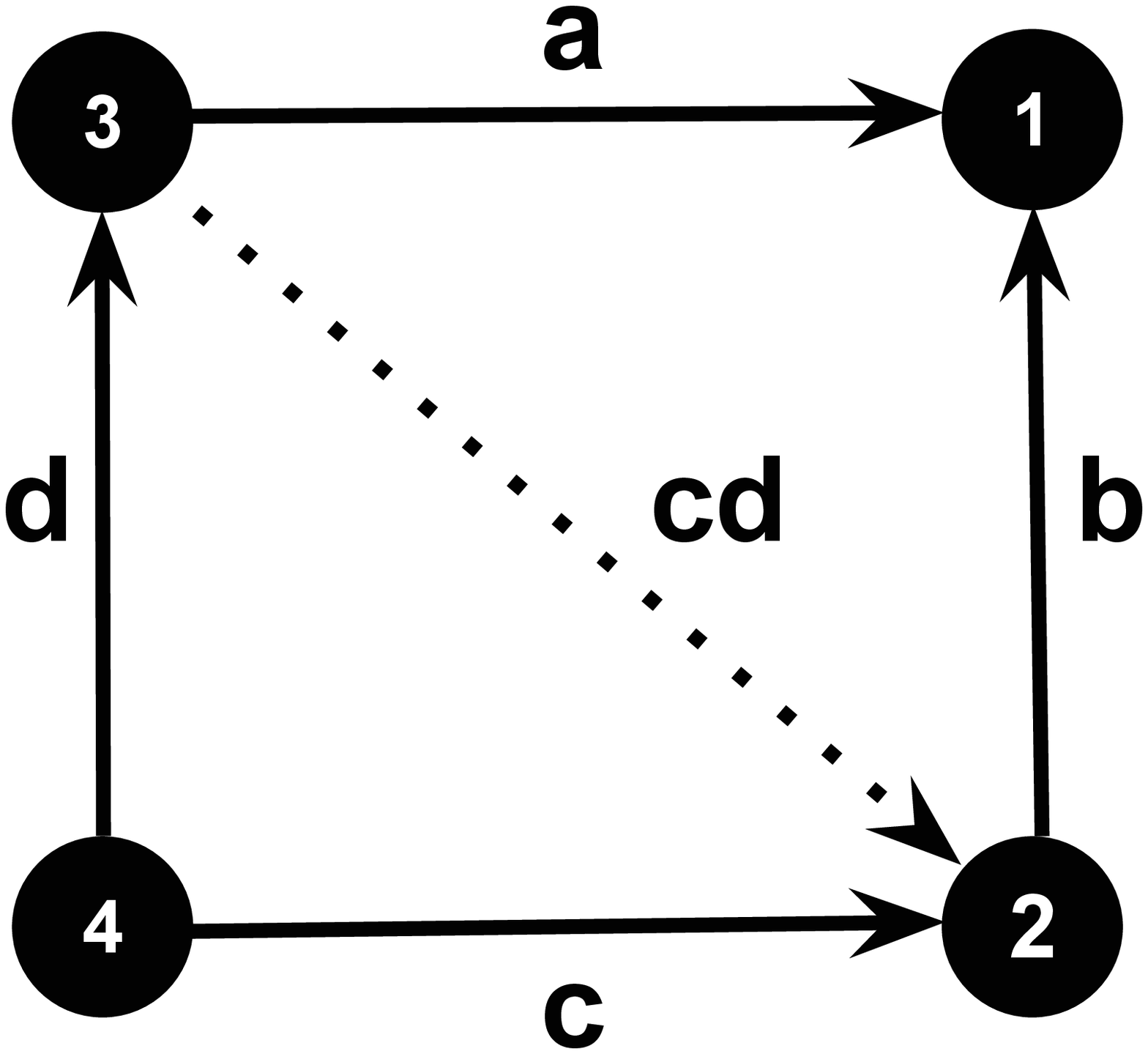}&\includegraphics[width=3.1cm]{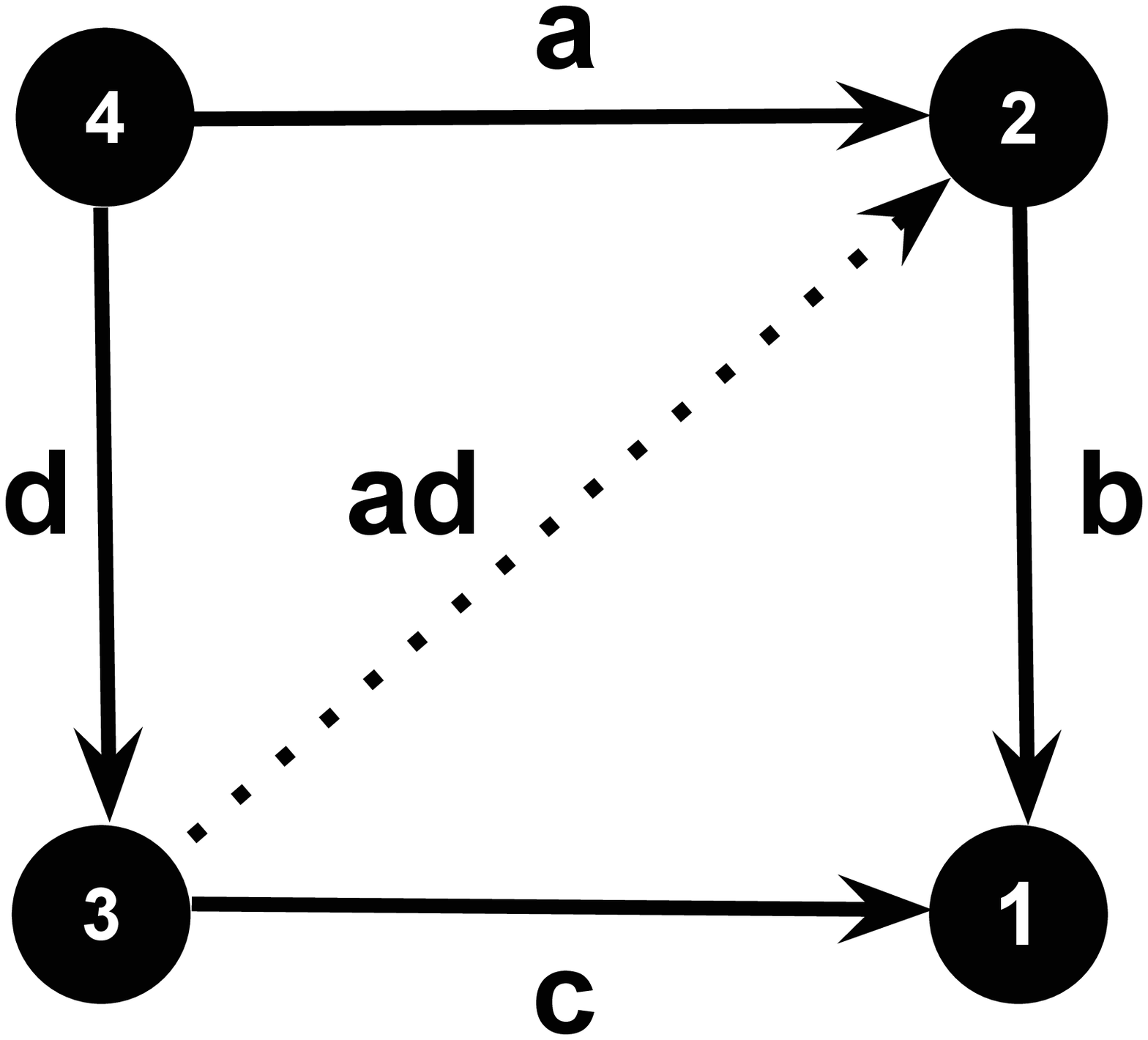}\\

(5)&(6)&(7)&(8)\\
& & & \\
\includegraphics[width=3.1cm]{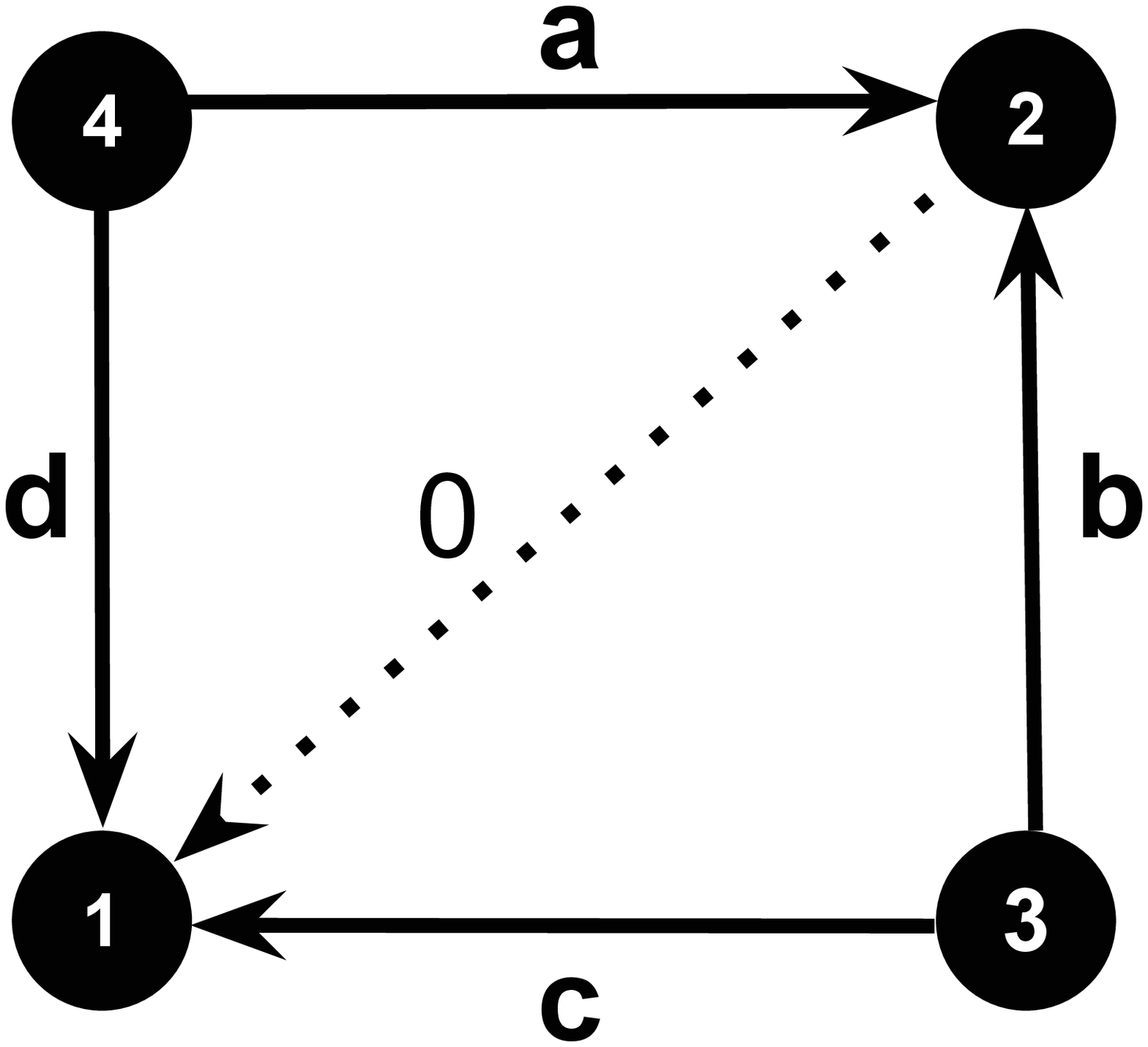}&\includegraphics[width=3.1cm]{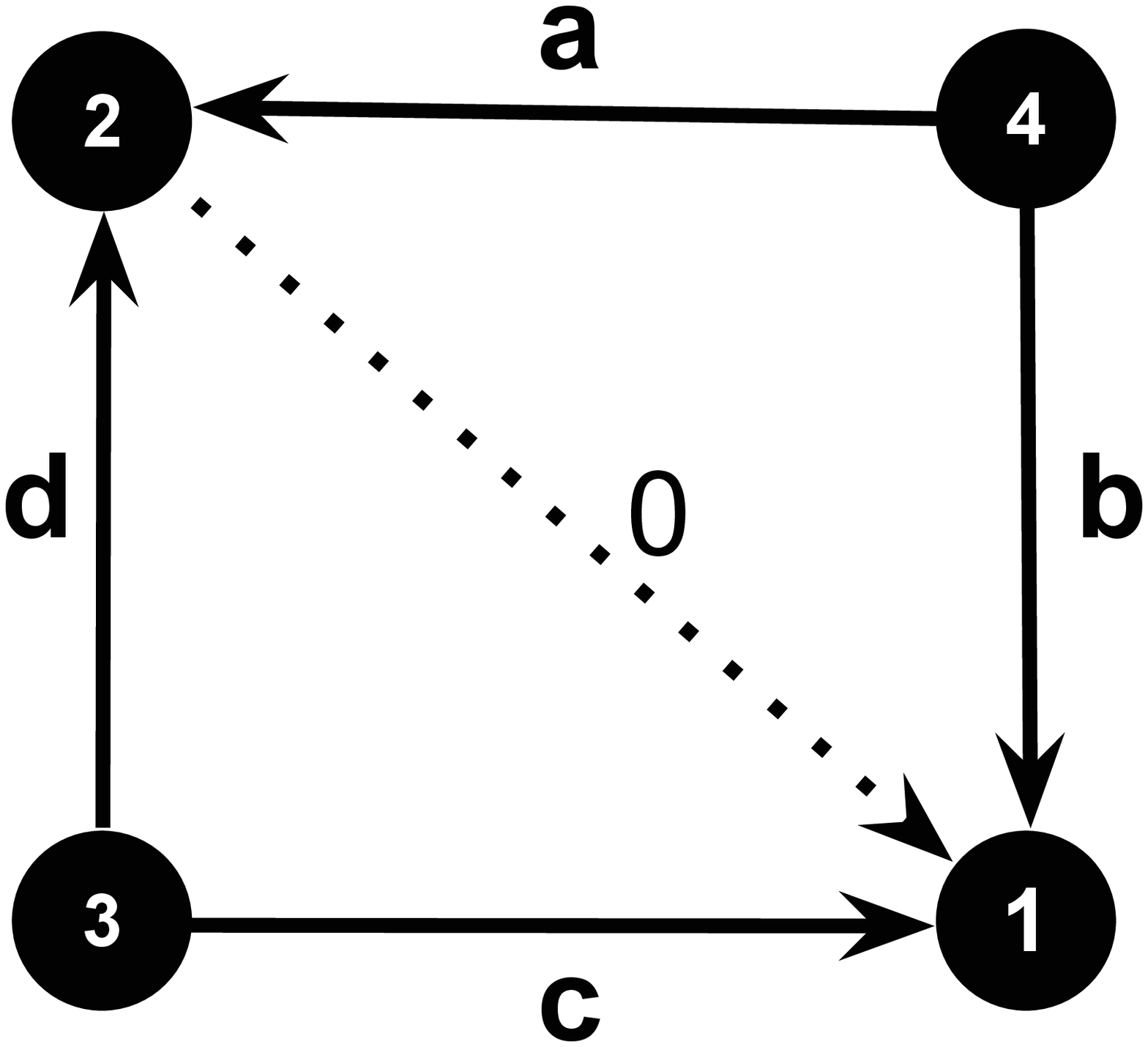}&  &\\
(9)&(10)&&
\end{tabular}
\end{center}
\label{table:dC4}
\end{table}%
\noindent
Table \ref{table:dC4} gives for each DAG version of $C_{4}$ the edges labeled with the corresponding entries of $\Gamma$. The dashed edges and their labels correspond to the missing edges and entries computed by using Equation \eqref{eq:DL_normal_repeat}. For example,  the DAG version in Table \ref{table:dC4} (1) corresponds to the partial matrix 
\[
\Gamma=
\left(
\begin{matrix}
1&c&*&d\\
c&1&b&*\\
*&b&1&a\\
d&*&a&1
\end{matrix}
\right).
\]
 Here the missing entry $\Sigma_{42}$ is computed as $\Sigma_{42}=\Sigma_{\pr(2), 2}=\Gamma_{43}\Gamma_{33}^{-1}\Gamma_{32}=ab$, which is the label on the dashed edge $4\dashrightarrow2$ in Table \ref{table:dC4} (1). Note that once the dashed edges are included all the DAGs in Table \ref{table:dC4} are perfect. Therefore, if the partial matrix over the corresponding perfect DAG in Table \ref{table:dC4} is partial positive definite, then by  Corollary \ref{cor:completion_per} the completion  in $\mathrm{PD}_{\D}$ (where $\D$ is the original DAG version of $C_{4}$) is guaranteed. The above reasoning allows us derive the required system of inequalities that sensure that $\Gamma$ can be completed in $\mathrm{PD}_{\D}$. For example, the partial matrix corresponding to the perfect DAG in Table \ref{table:dC4} (1) is 
\[
A = \left(
\begin{matrix}
1&c&*&d\\
c&1&b&ab\\
*&b&1&a\\
d&ab&a&1
\end{matrix}
\right) .
\]  

A simple calculation will show that the lower right $3 \times 3$ submatrix of $A$ is positive definite. Hence $A$ is  partial positive definite if and only if \:
$
\left(
\begin{matrix}
1&c&d\\
c&1&ab\\
d&ab&1
\end{matrix}
\right)\succ 0 
$. This is equivalent to requiring that 
\[
 f_{1}=(1-c^{2})(1-d^{2})-(ab-cd)^{2}>0.
\]
Similarly, we can show that the partial matrix corresponding to each of the perfect DAGs given in Table \ref{table:dC4} (2),  (3)  or  (4) respectively, is partial positive definite if and only if
\begin{align*}
f_{2}:&=(1-a^{2})(1-d^{2})-(bc-ad)^{2}>0,\\
 f_{3}:&=(1-a^{2})(1-b^{2})-(cd-ab)^{2}>0, \; \text{or}\\
 f_{4}:&=(1-b^{2})(1-c^{2})-(ad-bc)^{2}>0.
\end{align*}
It is easy to check from Table \ref{table:dC4}  that the inequalities obtained under the perfect DAGs given in Table \ref{table:dC4} (5), (6), (7) and (8)  are the same inequalities already listed above. The partial matrix corresponding to each of the perfect DAGs given in Table \ref{table:dC4} (9) or (10) respectively is,  partial positive definite if and only if
\begin{align*}
f_{5}:&=\min\left((1-b^{2})(1-c^{2})-(bc)^{2}, (1-a^{2})(1-d^{2})-(ad)^{2}\right)>0, \; \text{or} \\
f_{6}:&=\min\left((1-a^{2})(1-b^{2})-(ab)^{2}, (1-c^{2})(1-d^{2})-(cd)^{2}\right)>0.
\end{align*}
Now consider the following choices for $a, b, c, d$:  $a=0.6$,\; $b=0.9$,\;  $c=0.1$,\;  and \; $d=0.9$. Then we have
 $f(0.6, 0.9,  0.1,  0.9)= 0.3324$,  but
\[
\begin{array}{ll}
f_{1}(0.6, 0.9, 0.1 ,  0.9)=  -0.01&f_{2}(0.6, 0.9, 0.1,  0.9)=  -0.08\\
f_{3}(0.6, 0.9, 0.1, 0.9)=-0.08&f_{4}(0.6, 0.9, 0.1,  0.9)=-0.01\\
f_{5}(0.6, 0.9, 0.1, 0.9)=-0.17 & f_{6}(0.6, 0.9, 0.1, 0.9)=-0.17.
\end{array}
\]
Hence when the above values for $a, b, c, d$ are substituted in $\Gamma$ we obtain a partial positive definite matrix that cannot be completed in $\mathrm{PD}_{\D}$ for any DAG version  $\D$  of   $C_{4}$, although it can be completed in $\mathrm{PD}_{C_{4}}$. 
\section{Computing the inverse and determinant of the completion of an Incomplete matrix}\label{sec:det_inv}

In this section we give closed form expressions for the inverse and the determinant of a completed matrix $\mathrm{PD}_{\D}$ as a function of only the elements of the corresponding partial matrix. First we need the following notion for undirected graphs.\\
\begin{definition}
Let $\G=(V,\mathscr{V})$ be an arbitrary undirected graph.
\begin{itemize}
\renewcommand{\labelitemi}{$(i)$}
\item  For three disjoint subsets $A, B$ and $S$ of $V$ we say that $S$ separates $A$  from $B$ in $\G$ if 
 every path from a vertex in $A$ to a vertex in $B$ intersects a vertex in $S$.
 \renewcommand{\labelitemi}{$(ii)$}
 \item  Let $\Gamma$ be a $\G$-partial matrix. The zero-fill-in of $\Gamma$ in $\G$, denoted by $\left[\Gamma\right]^{V}$,
 is a $|V|\times |V|$ matrix $T$  such that
 \[
 T_{ij}=
 \begin{cases}
 \Gamma_{ij}& \text{if $\{i,j\}\in \mathscr{V}$,}\\
 0&\text{otherwise.}
  \end{cases} 
  \]
  \end{itemize}
 \end{definition}
\noindent
We now invoke a lemma required in the proof of the main result in this section.
\begin{lemma}\label{lem:inv_det}
Let $\D=(V, \mathscr{E})$ be an arbitrary DAG and let $\G=(V,\mathscr{V})$ denote the undirected version of $\D$. Let
$\Sigma\in \mathrm{PD}_{\D}$ and  let $(A, B, S)$ be a partition of $V$  such that $S$ separates $A$ from $B$ 
in $\D^{\mathrm{m}}$.  Then we have\\
~$(a)\quad \Sigma^{-1}=\left[ \left(\Sigma_{A\cup S}\right)^{-1}\right]^{V}+\left[ \left(\Sigma_{B\cup S}\right)^{-1}\right]^{V}
-\left[ \left(\Sigma_{ S}\right)^{-1}\right]^{V}$ and\\
~$(b)\quad \det(\Sigma^{-1})=\dfrac{\det(\Sigma_{S})}{\det(\Sigma_{A\cup S})\det(\Sigma_{B\cup S})}.$
\end{lemma}
\begin{proof}
Since by Lemma \ref{lem:wermuth} $\mathrm{PD}_{\D}\subseteq \mathrm{PD}_{\D^{\mathrm{m}}}$ the proof directly follows from Lemma 5.5 in \cite{Lauritzen1996}.
\end{proof}
\begin{proposition}\label{prop:inv_det}
Let $\Gamma$ be a partial positive definite matrix in $\mathrm{Q}_{\D}$ that can be completed to a positive 
definite matrix $\Sigma$ in $\mathrm{PD}_{\D}$. Then we have:\\
$~(i)\quad
 \Sigma^{-1}=\sum_{i=1}^{p}\left(\left[\left(\Sigma_{\fa(i)}\right)^{-1}\right]^{V}-\left[\left(\Sigma_{\pa(i)}\right)^{-1}\right]^{V}\right)$\:\:and\\
$~(ii)\quad\det(\Sigma^{-1})=\dfrac{\prod_{i=1}^{p} \det(\Sigma_{\pa(i)})}{\prod_{i=1}^{p} \det(\Sigma_{\fa(i)})}=\prod_{i=1}^{p}\Sigma_{ii|\pa(i)}^{-1}$.
\begin{proof}
Suppose that by mathematical induction the assertion of the proposition holds for any DAG with number of vertices less than $p$. We proceed to prove the proposition for $\D$ with $p$ vertices. The case $p=1$ holds trivially. Thus assume that $p>1$. Let  $V_{[1]}$ denote $V\setminus\{1\}$. The triple $\left(V\setminus \fa(1), \{1\}, \pa(1)\right)$ is a partition of $V$, and $\pa(1)$ separates  $\{ 1\}$ from  $V\setminus \fa(1)$ in $\D^{\mathrm{m}}$. Thus by Lemma \ref{lem:inv_det}, 
\begin{equation}\label{eq:induction_step1}
\Sigma^{-1}=\left[\left(\Sigma_{V\setminus \{1\}}\right)^{-1}\right]^{V} + \left[  \left(\Sigma_{\fa(1)}\right)^{-1} \right]^{V}-\left[\left(\Sigma_{\pa(1) }\right)^{-1} \right]^{V}. 
\end{equation}
Invoking the same notation used in the proof of Theorem \ref{thm:existence_completion}, note that the partial matrix $\Gamma^{[1]}$ is positive definite over $\D_{[1]}$ with positive completion  $\Sigma_{V\setminus\{1\}}$ in $\mathrm{PD}_{\D_{[1]}}$. By the induction hypothesis applied to $\Psi=\Sigma_{V\setminus\{1\}}$ we have
\begin{equation}\label{eq:induction_step}
\left(\Sigma_{V\setminus \{1\}}\right)^{-1}=\Psi^{-1}=\sum_{i=2}^{p}\left(\left[\left(\Psi_{\fa(i)}\right)^{-1}\right]^{V_{[1]}}-\left[\left(\Psi_{\pa(i)}\right)^{-1}\right]^{V_{[1]}}\right).
\end{equation}
Since $\D_{[1]}$ is an ancestral subgraph of $\D$, it is clear that for each $i\in V\setminus\{1\}$ we have
\[
\Psi_{\fa(i)}=\Sigma_{\fa(i)}\:\:\text{and}\:\:  \Psi_{\pa(i)}=\Sigma_{\pa(i)}.
\]
By replacing these in Equation \eqref{eq:induction_step} and zero-fill-in in $\D^{\mathrm{u}}$ we obtain
\begin{equation}\label{eq:Sigma}
\left[\left(\Sigma_{V\setminus \{1\}}\right)^{-1}\right]^{V}=\sum_{i=2}^{p}\left(\left[\left(\Sigma_{\fa(i)}\right)^{-1}\right]^{V}-\left[\left(\Sigma_{\pa(i)}\right)^{-1}\right]^{V}\right).
\end{equation}
Finally, substituting Equation \eqref{eq:Sigma} into Equation \eqref{eq:induction_step1} yields the formula in part $(i)$. Part $(ii)$ follows similarly by using  part $(2)$ in Lemma \ref{lem:inv_det} .
\end{proof}
\end{proposition}

\begin{Rem}
We note that Part $(ii)$ of Proposition \ref{prop:inv_det} can also be proved using probabilistic arguments (see \cite{Andersson1998, Bendavid2011}).
\end{Rem}

\begin{figure}[htbp]
\begin{center}
\includegraphics[width=3.1cm]{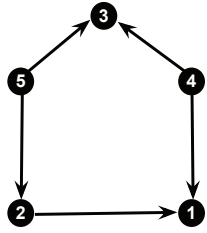}
\caption{The DAG from Example \ref{ex:inv_det} for which $\Sigma^{-1}$ is computed using Proposition \ref{prop:inv_det}}
\label{fig:inv_ex}
\end{center}
\end{figure}
\begin{Ex}\label{ex:inv_det}
Let $\D$ be the DAG given in Figure \ref{fig:inv_ex}.\\
~$(a)$~ Suppose that the  $\D$-partial matrix 
\[
\Gamma=
\left(
\begin{matrix}
1&\Sigma_{12}&*&\Sigma_{14}&*\\
\Sigma_{21}&1&*&*&\Sigma_{25}\\
*&*&1&\Sigma_{34}&\Sigma_{35}\\
\Sigma_{41}&*&\Sigma_{43}&1&*\\
*&\Sigma_{52}&\Sigma_{53}&*&1
\end{matrix}
\right)
\]
can be completed to a positive definite matrix $\Sigma$ in $\mathrm{PD}_{\D}$. By applying part $(i)$ of Proposition \ref{prop:inv_det}
we have
\begin{align*}
\Sigma^{-1}&=\left[(\Sigma_{\{1,2,4\}})^{-1}\right]^{V}+\left[(\Sigma_{\{2,5\}})^{-1}\right]^{V}+
\left[(\Sigma_{\{3,4,5\}})^{-1}\right]^{V}+\left[\Sigma_{44}^{-1}\right]^{V}\\
&+\left[\Sigma_{55}^{-1}\right]^{V}-\left[(\Sigma_{\{2,4\}})^{-1}\right]^{V}-\left[\Sigma_{55}^{-1}\right]^{V}
-\left[(\Sigma_{\{4,5\}})^{-1}\right]^{V}.
\end{align*}
Note that all the entries of the matrices involved in this expression are given in $\Gamma$, except for $\Sigma_{54}$  and $\Sigma_{42}$. 
By using Equation \eqref{eq:DL_normal_repeat} it is easy to check that $\Sigma_{54}=\Sigma_{42}=0$. Hence, 
\begin{align*}
\Sigma^{-1}&=
\left[
\left(
\begin{matrix}
1& \Sigma_{12}&\Sigma_{14}\\
\Sigma_{21}&1&0\\
\Sigma_{41}&0&1
\end{matrix}
\right)^{-1}
\right]^{V}+
\left[
\left(
\begin{matrix}
1&\Sigma_{25}\\
\Sigma_{52}&1
\end{matrix}
\right)^{-1}
\right]^{V}+
\left[
\left(
\begin{matrix}
1& \Sigma_{34}&\Sigma_{35}\\
\Sigma_{43}&1&0\\
\Sigma_{53}&0&1
\end{matrix}
\right)^{-1}
\right]^{V}\\
&+
\left(
\begin{matrix}
0&0&0&0&0\\
0&0&0&0&0\\
0&0&0&0&0\\
0&0&0&1&0\\
0&0&0&0&0
\end{matrix}
\right)-
\left(
\begin{matrix}
0&0&0&0&0\\
0&1&0&0&0\\
0&0&0&0&0\\
0&0&0&1&0\\
0&0&0&0&0
\end{matrix}
\right)-
\left(
\begin{matrix}
0&0&0&0&0\\
0&0&0&0&0\\
0&0&0&0&0\\
0&0&0&1&0\\
0&0&0&0&1
\end{matrix}
\right)\\
&=
\frac{1}{1-\Sigma_{12}^{2}-\Sigma_{14}^{2}}\left(
\begin{matrix}
1&-\Sigma_{12}  &0 &-\Sigma_{14} &0 \\
-\Sigma_{12}&1-\Sigma_{14}^{2}  &0 &\Sigma_{12}\Sigma_{14} &0 \\
0& 0 & 0& 0&0 \\
-\Sigma_{14}&\Sigma_{12}\Sigma_{14}  & 0&1-\Sigma_{12}^{2} &0 \\
0&0  &0 &0 &0 
\end{matrix}
\right)
+\frac{1}{1-\Sigma_{25}^{2}}
\left(
\begin{matrix}
0&0&0&0&0\\
0&1&0&0&-\Sigma_{25}\\
0&0&0&0&0\\
0&0&0&0&0\\
0&-\Sigma_{25}&0&0&1
\end{matrix}
\right)\\
&+
\frac{1}{1-\Sigma_{34}^{2}-\Sigma_{35}^{2}}
\left(
\begin{matrix}
0&0  & 0&0 &0 \\
0 &  0&0 &0 &0 \\
0 &0 &1&-\Sigma_{34}  & -\Sigma_{35} \\
0 & 0 & -\Sigma_{34} &1-\Sigma_{35} ^{2} &\Sigma_{34} \Sigma_{35}  \\
0 & 0 & -\Sigma_{35} &\Sigma_{34} \Sigma_{35}  &1-\Sigma_{34} ^{2} 
\end{matrix}
\right)
+\left(
\begin{matrix}
0 & 0 & 0& 0&0 \\
 0&-1  &0 &0 &0 \\
0 &0  &0 &0 &0 \\
0 &0  &0 &-1 &0 \\
0 &0  & 0&0 &-1 
\end{matrix}
\right).
\end{align*}
By combining these terms into one matrix we have $\Sigma^{-1}$ is equal to:
\begin{align*}
&
\left(
\begin{matrix}
 \frac{1}{1-\Sigma_{12}^{2}-\Sigma_{14}^{2}}& \frac{-\Sigma_{12}}{1-\Sigma_{12}^{2}-\Sigma_{14}^{2}}&0 &\frac{-\Sigma_{14}}{1-\Sigma_{12}^{2}-\Sigma_{14}^{2}} &0 \\
  \frac{-\Sigma_{12}}{1-\Sigma_{12}^{2}-\Sigma_{14}^{2}}&\frac{1-\Sigma_{14}^{2}}{1-\Sigma_{12}^{2}-\Sigma_{14}^{2}}+\frac{1}{1-\Sigma_{25}^{2}}-1 &0 &\frac{\Sigma_{12}\Sigma_{14}}{1-\Sigma_{12}^{2}-\Sigma_{14}^{2}} & \frac{-\Sigma_{25}}{1-\Sigma_{25}^{2}}\\
0  & 0&\frac{1}{1-\Sigma_{34}^{2}-\Sigma_{35}^{2}} & \frac{-\Sigma_{34}}{1-\Sigma_{34}^{2}-\Sigma_{35}^{2}}&\frac{-\Sigma_{35}}{1-\Sigma_{34}^{2}-\Sigma_{35}^{2}} \\
    \frac{-\Sigma_{14}}{1\Sigma_{12}^{2}-\Sigma_{14}^{2}}&\frac{\Sigma_{12}\Sigma_{14}}{1-\Sigma_{12}^{2}-\Sigma_{14}^{2}} &\frac{-\Sigma_{34}}{1-\Sigma_{34}^{2}-\Sigma_{35}^{2}} & \frac{1-\Sigma_{12}^{2}}{1-\Sigma_{12}^{2}-\Sigma_{14}^{2}}+\frac{1-\Sigma_{35}^{2}}{1-\Sigma_{34}^{2}-\Sigma_{35}^{2}}-1&\frac{\Sigma_{34}\Sigma_{35}}{1-\Sigma_{34}^{2}-\Sigma_{35}^{2}} \\
  0   & \frac{-\Sigma_{25}}{1-\Sigma_{25}^{2}}&\frac{-\Sigma_{35}}{1-\Sigma_{34}^{2}-\Sigma_{35}^{2}} &\frac{\Sigma_{34}\Sigma_{35}}{1-\Sigma_{34}^{2}-\Sigma_{35}^{2}} &\frac{1-\Sigma_{34}^{2}}{1-\Sigma_{34}^{2}-\Sigma_{35}^{2}}+\frac{1}{1-\Sigma_{25}^{2}}-1 
\end{matrix}
\right).
\end{align*}
\medskip\noindent
Similarly, using part $(ii)$ of Proposition \ref{prop:inv_det} we have
\[
\det(\Sigma^{-1})=\left[(1-\Sigma_{12}^{2}-\Sigma_{14}^{2})(1-\Sigma_{25}^{2})(1-\Sigma_{34}^{2}-\Sigma_{35}^{2})\right]^{-1}.
\]
$(b)$~ Let us apply the computation of the inverse matrix in part $(a)$ to the  following specific $\D$-partial matrix
\[
\Gamma=
\left(
\begin{matrix}
4 &  -2  &  *   &1   & *\\
  -2 &   2   & * &   * &  -1\\
   * &   *  &  3  &  1  & -1\\
  1  &  *  &1  &  1   & *\\
  *  & -1  & -1&    *   &1
  \end{matrix}
\right).
\]
Using the completion process in Proposition \ref{prop:completion_in_PDD} one can check that the completion of $\Gamma$ in $\mathrm{PD}_{\D}$ is given as
\[
\Sigma=
\left(
\begin{matrix}
 4 &  -2  &  0    &1   & 1\\
  -2 &   2   & 1 &   0 &  -1\\
   0 &   1  &  3  &  1  & -1\\
  1  &  0    &1  &  1   & 0\\
  1  & -1  & -1&    0    &1
 \end{matrix}
\right).
\]
From this we obtain
\begin{equation}\label{eq:Sigma_inv}
\Sigma^{-1}=
\left(
\begin{matrix}
 1   & 1   & 0  & -1  &  0\\
 1   & 2  &  0  & -1   & 1\\
 0   & 0   & 1  & -1   & 1\\
   -1 &  -1  & -1   & 3  & -1\\
 0   & 1   & 1 &  -1 &   3

\end{matrix}
\right).
\end{equation}
However, without completing  $\Sigma$ and with less computation, we can compute $\Sigma^{-1}$ using the results in Proposition \ref{prop:inv_det}. For this, we apply part $(a)$ above. First let $D=\mathrm{diag}( 4,2,3,1,1)$ and $\Gamma_{0}=D^{-\frac{1}{2}}\Gamma D^{-\frac{1}{2}}$. Thus
\[
\Gamma_{0}=
\left(
\begin{matrix}
  1& -0.71&  *&0.50&  *\\
 -0.71&  1&*& *& -0.71\\
*& *&  1&0.58& -0.58\\
0.50&  *&  0.58 &1& *\\
* &-0.71& -0.58 &*&1
 \end{matrix}
 \right).
\]
If $\Sigma_{0}$ denotes  the completion of $\Gamma_{0}$  in $\mathrm{PD}_{\D}$, then by using part $(a)$ we obtain
\[
\Sigma_{0}^{-1}=
\left(
\begin{matrix}
4 &2.83 & 0&-2 & 0\\
2.83 & 4 &  0&-1.41 & 1.41\\
 0&0 &  3&-1.73  &1.73\\
-2 &-1.41& -1.73 & 3 &-1\\
 0&1.41&  1.73& -1  &3\\
\end{matrix}
\right).
\]
From this we can obtain $\Sigma^{-1}=D^{-\frac{1}{2}}\Sigma_{0}^{-1}D^{-\frac{1}{2}}$, which as expected turns out to be same matrix as the one given in Equation \eqref{eq:Sigma_inv}.
\end{Ex}
Example  \ref{ex:inv_det} demonstrates that the inverse and the determinant of the completion of a partial matrix  $\Gamma$ in $\mathrm{PD}_{\D}$ can be computed using the expressions given in Proposition \ref{prop:inv_det}, and generally, avoid recourse to the whole completion process. This is especially so when $\D$ is a perfect DAG. Since in this case $\Sigma_{\fa(i)}$, and consequently  $\Sigma_{\pa(i)}$, are already blocks of the partial matrix $\Gamma$, the computations can be carried out without recourse to the completion process in Proposition \ref{prop:completion_in_PDD}. This fact about perfect DAGs can also be deduced from their relationship to undirected decomposable graphs (see Lemma 5.5 in \cite{Lauritzen1996}).

\vspace{0.3cm}

\noindent {\it Acknowledgments:} Ben-David was supported in part by a seed grant from the Cardiovascular Institute, Stanford University School of Medicine and by the National Science Foundation under Grant No. DMS-CMG-1025465. Rajaratnam was supported in part by the National Science Foundation under Grant Nos. DMS-0906392, DMS-CMG-1025465, AGS-1003823, DMS-1106642 and grants NSA H98230-11-1-0194, DARPA-YFA N66001-11-1-4131, and SUWIEVP10-SUFSC10-SMSCVISG0906. 

\vspace{0.3cm}


\bibliographystyle{plain}








\end{document}